\documentclass[8pt]{article}
\usepackage{amsfonts}
\usepackage[pdftex]{graphicx}
\usepackage{amsmath}
\usepackage{amssymb}
 \usepackage{epsfig}
  \usepackage{epstopdf}
 \usepackage{subcaption}
\usepackage{wrapfig}
\usepackage{booktabs}
\usepackage{multirow}
\usepackage{tabularx}
\usepackage{a4}
\usepackage[driverfallback=pdfmx,,colorlinks,bookmarksopen,bookmarksnumbered,citecolor=blue, urlcolor=blue]{hyperref}

\newenvironment{keywords}{
\list{}{\advance\topsep by0.35cm\relax\small
\leftmargin=1cm
\itemindent\listparindent
\rightmargin\leftmargin}\item[\hskip\labelsep
\bfseries Keywords:]}
{\endlist}
\newtheorem{theorem}{Theorem}[section]

\newtheorem{lemma}[theorem]{Lemma}
\newtheorem{proof}[theorem]{Proof}
\newtheorem{proposition}[theorem]{Proposition}
\newtheorem{definition}[theorem]{Definition}

\newtheorem{remark}[theorem]{Remark}

\begin{document}

\title{Long time dynamics of a three-species food chain model with Allee effect in the top predator}

\maketitle

\centerline{\scshape Rana D. Parshad, Emmanuel Quansah and Kelly Black}
\medskip
{\footnotesize
 \centerline{ Department of Mathematics,}
 \centerline{Clarkson University,}
   \centerline{ Potsdam, New York 13699, USA.}
   } 

   \medskip
\centerline{\scshape Ranjit K. Upadhyay and S.K.Tiwari}
{\footnotesize
 \centerline{Department of Applied Mathematics,}
 \centerline{Indian School of Mines,}
   \centerline{ Dhanbad 826004, Jharkhand, India.}
}

\medskip
\centerline{\scshape Nitu Kumari}
{\footnotesize
 \centerline{School of Basic Sciences,}
 \centerline{Indian Institute of Technology Mandi,}
 \centerline{ Mandi, Himachal Pradesh 175 001, India.}
}

\begin{abstract}
 The Allee effect is an important phenomenon in population biology characterized by positive density dependence, that is a positive correlation between population density and individual fitness. However, the effect is not well studied in multi-level trophic food chains. We consider a ratio dependent spatially explicit three species food chain model, where the top predator is subjected to a strong Allee effect. 
 We show the existence of a global attractor for the the model, that is upper semicontinuous in the Allee threshold parameter $m$. To the best of our knowledge this is the first robustness result, for a spatially explicit three species food chain model with an Allee effect.
Next, we numerically investigate the decay rate to a target attractor, that is when $m=0$, in terms of $m$.
We find decay estimates that are $\mathcal{O}(m^{\gamma})$, where $\gamma$ is found explicitly.
Furthermore, we prove various overexploitation theorems for the food chain model, showing that overexploitation has to be driven by the middle predator. In particular overexploitation is not possible without an Allee effect in place. We also uncover a rich class of Turing patterns in the model which depend significantly on the Allee threshold parameter $m$. Our results have potential applications to trophic cascade control, conservation efforts in food chains, as well as Allee mediated biological control.
\end{abstract}


\section{\textbf{Introduction}}

Interactions of predator and prey species are ubiquitous in spatial ecology. Therein a predator or a ``hunting" organism, hunts down and attempts to kill a prey, in order to feed. 
Food webs, which comprise all of the predator-prey interactions in a given ecosystem, are inherently more complex. The food chains or linear links of these food webs in real ecosystems, have multiple levels of predator-prey interaction, across various trophic levels. 
To better understand natural ecosystems, with multiple levels of trophic interaction, where predators and prey alike, disperse in space in search of food, mates and refuge, a natural starting point is to deviate from the classical predator-prey two species models, and investigate spatially explicit three species food chains.
Such models are appropriate to model populations of generalist predators, predating on a specialist predator, which in turn predates on a prey species. Also, they can model two specialist predators competing for a single prey \cite{M93, UR97}. The spatial component of ecological interaction has been identified as an important factor in how ecological communities are shaped, and thus the effects of space and spatially dispersing populations, via partial differential equations (PDE)/spatially explicit models of three and more interacting species, have been very well studied \cite{Gilligan1998, M93, okubo2001diffusion, sen2012bifurcation, PK14}. In most of these models, the prey is regulated or ``inhibited" from growing to carrying capacity due to predation by the predator. Whereas loss in the predator is due to death or intraspecific competition terms. 
However, there are various other natural self regulating mechanisms in a population of predators (or prey).
For instance, much research in two species models, has focused on one such mechanism: the so called Allee effect. However, less attention has been paid to this mechanism in the three species case. 
This effect, named after the ecologist Walter Clyde Allee, can occur whenever fitness of an individual in a small or sparse population, decreases as the population size or density does \cite{drake2011allee}. Since the pioneering work of Allee \cite{allee1931co,allee1949principles}, Allee dynamics has been regarded as one of the central issues in the population and community ecology. 

The effect can be best understood by the following equation for a single species $u$,

\begin{equation} 
\label{eq:af1}
\frac{d u}{d t} = u(u-m)(1-u/K),
\end{equation}

essentially a modification to the logistic equation. Here $u(t)$ the state variable, represents the numbers of a certain species at a given time $t$, $K$ is the carrying capacity of the environment that $u$ resides in, and $m$ is the Allee threshold, with $m < K$. A strong Allee effect occurs if $m > 0$, \cite{van2007heteroclinic}. This essentially means that if $u$ falls below the threshold population $m$, its growth rate is negative, and the species will go extinct \cite{pal2014bifurcation}. If $m < 0$, then there is a weak Allee effect, and we have a compensatory growth function \cite{liermann2001depensation, bioeconomics1990optimal, clark2006worldwide}. For our purposes we assume a strong Allee effect is in place, that is $m > 0$.
Note, $u^{*}=0,K$ are stable fixed points for \eqref{eq:af1}, and $u^{*}=m$ is unstable. Thus dynamically speaking the global attractor, which is the repository of all the long time dynamics of for \eqref{eq:af1} is $[0,K]$. An interesting question can now be posed: what happens to this attractor as $m \rightarrow 0$? Ecologically speaking, this is asking: what happens to the species $u$ as the Allee threshold $m$ is decreased? When $m=0$ the only stable fixed point is $u^{*}=K$, with $u^{*}=0$ being half stable. Thus the global attractor is now reduced to a single point $K$. What we observe is that the difference between having a slight Allee effect (say $0 < m <<1$) and having no Allee effect ($m=0$), can \emph{change completely} what the global attractor of \eqref{eq:af1} is. That is to say, the global attractor changes from an entire set $[0,K]$, to a single fixed point $K$.
From an ecological point of view, this says as long as there is a slight Allee effect there is an extinction risk for $u$, but without one \emph{there is none}, and $u$ will always grow to carrying capacity.
Proving that an attractor $\mathcal{A}_{m}$ approaches a target attractor $\mathcal{A}_{0}$, in a continuous sense, in the case of spatially explicit/PDE models requires a fair amount of work. This involves making several estimates of functional norms, independent of the parameter $m$ \cite{SY02}. If this can be proven however, the attractor $\mathcal{A}_{m}$ is said to be \emph{robust} at $m=0$. However, to the best of our knowledge there are no robustness results in three species food chain models, where the parameter of interest is the Allee threshold $m$. Note, unless a systems dynamics are robust, there is no possibility to capture the same in a laboratory experiment or natural setting.

The critical importance of the Allee effect has widely been realized in the conservation biology: That is: it is most likely to increase the extinction risk of low-density populations. This is also known as critical depensation in fisheries sciences. Essentially individual species population must surpass the threshold $m$ to grow. Hence a small introduced population under strong Allee effect can only succeed if it is faced with favorable ecological conditions, or experiences rapid adaptive evolution, or simply has good luck, so that it can surpass this critical threshold. 
There are several factors that might cause an Allee effect. These include difficulty of finding mates at low population densities, inbreeding depression and environmental conditioning\cite{tobin2011exploiting}. 
The Allee effect can be regarded not only as a suite of problems associated with rarity, but also as the basis of animal sociality \cite{stephens1999consequences}.
Petrovskii et al. \cite{SPM2002} found that a deterministic system with Allee effect, can induce patch invasion. Morozov et al. \cite{AMP2004} found that the temporal population oscillations can exhibit chaotic dynamics even when the distribution of the species in the space was regular. Also, Sharma and Samanta \cite{sharma2014ratio} developed a ratio-dependent predator-prey model with disease in prey, as well as an Allee effect in prey. Furthermore, Invasion biologists have recently attempted to consider Allee effects as a benefit in limiting establishment of an invading species\cite{tobin2011exploiting}, which without the effect would possibly run amuck, an cause excessive damages to native species because of predation and competition.

On the same lines, a phenomenon that occurs in large food chains, is the excessive harvesting/predation of certain species in the chain, via the predators in the trophic level above them. In many aquatic food chains, this can lead to trophic cascades  \cite{C85}, and is among the major activities that is threatening global biodiversity. Formally, overexploitation refers to the phenomenon where excessive harvesting of a species can result in its extinction. Mathematically, for a two species predator-prey system, one can prove an overexploitation type theorem, if one shows for a large enough initial density of the predator, the prey will be predated on till extinction. In this case, If the predator does not have an alternate food source, it will also subsequently go extinct. Although there is a fair amount of literature on this in two species models \cite{shi11}, it is less studied in three species models. In particular, the Allee effect itself on overexploitation, in multi-trophic level food chains, has not been sufficiently explored.

Our primary goal in the current manuscript is to propose and analyze a reaction-diffusion three species food chain model, with an Allee effect in the top predator. In particular, we aim to investigate the link between the Allee threshold parameter $m$, and the longtime dynamics of the three species food chain, in terms of 

\begin{itemize}
\item Robustness of global attractors in terms of the Allee threshold parameter $m$.
\item Overexploitation phenomenon in the food chain model, as it is effected by the Allee threshold parameter $m$ .
\item Pattern formation and the effect on patterns that form in the food chain model due to the Allee threshold parameter $m$.
\end{itemize}

Our hope is that this model could be used as a feasible toy model, to better understand the following in multi-trophic level food chains, 

\begin{itemize}
\item Top-down pressure in multi-trophic level food chains leading to overexploitation phenomenon.

\item Trophic cascades and cascade control via an Allee effect.

 \item Conservation efforts in food webs mediated via an Allee effect.

\item Biological control of an invasive top predator in a food chain via Allee effect
\end{itemize}

 Thus, we consider a situation where a prey species $u$ serves as the only food for a specialist predator $v$ which is itself predated by a generalist top predator $r$. \\
This is a typical situation often seen in nature in various food chains \cite{P10}. The governing equations for populations $u$ and $v$ follow ratio dependent functional responses, and are modeled by the Volterra scheme i.e., the middle predator population dies out exponentially in the absence of its prey. There is recently much debate about functional responses used in ecology, and the ratio dependent response is considered to be more realistic than its Holling type counterparts \cite{A00}. The above situation is described via the following system of PDE. This is already nondimensionalised, see appendix \ref{app1} for the details of the nondimensionalisation. 

\begin{align}
&\frac{\partial u}{\partial t}= d_1\Delta u + u-u^{2}-w_{1}\frac{uv}{u+v}, \label{eq:x1}\\
&\frac{\partial v}{ \partial t}= d_2 \Delta v -a_{2}v+w_{2}\frac{uv}{u+v}-w_{3}\left(\frac{vr}{v+r}\right), \label{eq:x2}\\
&\frac{\partial r}{ \partial t} = d_3 \Delta r + r\big({{r-m}}\big)\bigg(c-\frac{w_4r} {v+D_3}\bigg).\label{eq:x3},
\end{align}

 the spatial domain for the above is $\Omega \subset \mathbb{R}^{n}$, $n=1,2,3$. $\Omega$ is assumed bounded, and we prescribe Neuman boundary conditions $\nabla u \cdot \textbf{n} = \nabla v \cdot \textbf{n} = \nabla r \cdot \textbf{n} = 0$. In the above $a_2$ is the intrinsic death rate of the specialist predator $v$ in the absence of its only food $u$, $c$ measures the rate of self-reproduction of the generalist predator $r$. $w_i's$ are the maximum values which per capita growth rate can attain. $D_{3}$ shows that the top predator $r$ is a generalist and can switch its food source in the absence of $v$. Also we model the top predator via a Leslie Gower scheme \cite{UR97}, and assume that it is subject to the Allee effect. The parameter $m$ is the Allee threshold or minimum of viable population level. In this work, we limit ourselves to the case where $m > 0$, that is we assume a strong Allee effect. 

In particular we ask: How does the Allee threshold $m$ affect the various dynamical aspects of model system \eqref{eq:x1}-\eqref{eq:x3}? To this end we list our primary findings.
\begin{itemize}
\item We show that there is a $(L^{2}(\Omega),H^{1}(\Omega))$ global attractor for \eqref{eq:x1}-\eqref{eq:x3}, via theorem \ref{thm:ga2}.

 \item  We show that this attractor is upper semi-continuous w.r.t the Allee threshold $m$, via theorem \ref{thm:gaus1sc}. That is the attractor $\mathcal{A}_{m}$ is robust at $m=0$. This result requires making estimates of various functional norms, independent of the parameter $m$, and then proving and using a chain of theorems that facilitate passing to the limit as $m \rightarrow 0$. The calculations are detailed, so we confine them to an appendix section \ref{app}.
 
 \item  The decay rate of the global attractor $\mathcal{A}_{m}, m>0$, for the system \eqref{eq:x1}-\eqref{eq:x3}, to a target attractor $\mathcal{A}_{0}$ (that is when $m=0$) is estimated in terms of $m$ numerically. We find a decay rate of the order $\mathcal{O}(m^{\gamma})$, where $\gamma$ is close to 1.

\item  We investigate overexploitation phenomenon in the model system \eqref{eq:x1}-\eqref{eq:x3}, showing that the Allee threshold $m$ can effect overexploitation in the system via theorems \ref{thm:ox}, \ref{thm:ox1} and lemma \ref{lem:ox2}. In particular overexploitation is \emph{not possible} without an Allee effect in place.

\item  Turing instability exists in the model system \eqref{eq:x1}-\eqref{eq:x3}, via theorem \ref{thm:tur1}. Furthermore, the Allee threshold $m$ has a significant impact on the type of Turing patterns that form, see figures  \ref{fig:turing1}, \ref{fig:turing2}, \ref{fig:turing3}, \ref{fig:turing4}, \ref{fig:turing5}, \ref{fig:Disp1}.

\end{itemize}

\section{Preliminary Estimates}

\subsection{Preliminaries}
We now present various notations and definitions that will be used frequently. 
The usual norms in the spaces $\mathbb{L}^{p}(\Omega )$, $\mathbb{L}^{\infty
}(\Omega )$ and $\mathbb{C}\left( \overline{\Omega }\right) $ are
respectively denoted by

\begin{equation*}
\left\Vert u\right\Vert _{p}^{p}\text{=}\frac{1}{\left\vert \Omega
\right\vert }\int\limits_{\Omega }\left\vert u(x)\right\vert ^{p}dx,
\end{equation*}%
\begin{equation*}
\left\Vert u\right\Vert _{\infty }\text{=}\underset{x\in \Omega }{max}%
\left\vert u(x)\right\vert .
\end{equation*}

 We define the following phase spaces
\begin{equation*}
    H= L^{2}(\Omega)\times L^{2}(\Omega) \times L^{2}(\Omega) ,
\end{equation*}

\begin{equation*}
    E= H^{1}(\Omega)\times H^{1}(\Omega) \times H^{1}(\Omega),
\end{equation*}

\begin{equation*}
    X= H^{2}(\Omega)  \times H^{2}(\Omega) \times H^{2}(\Omega) .
\end{equation*}

We now recall the following lemma.
\begin{lemma}[Uniform Gronwall Lemma]
\label{lem:gronwall}
Let $\beta, \zeta$, and $h$ be nonnegative functions in $L^{1}_{loc}[0,\infty;\mathbb{R})$. Assume that $\beta$
is absolutely continuous on $(0,\infty)$ and the following differential inequality is satisfied:
\begin{equation}
\frac{d \beta}{dt} \leq \zeta \beta + h, \ \mbox{for} \ t>0.
\end{equation}
If there exists a finite time $t_{1} > 0$ and some $q > 0$ such that
\begin{equation}
\int^{t+q}_{t}\zeta(\tau) d\tau \leq A,  \ \int^{t+q}_{t}\beta(\tau) d\tau \leq B, \ \mbox{and} \ \int^{t+q}_{t} h(\tau) d\tau \leq C,
\end{equation}
for any $t > t_{1}$, where $A, B$, and $C$ are some positive constants, then
\begin{equation}
\beta(t) \leq \left(\frac{B}{q}+C\right)e^{A}, \ \mbox{for \ any} \ t > t_{1}+q.
\end{equation}
\end{lemma}

\subsection{Uniform $L^{2}(\Omega)$ and $H^{1}(\Omega)$ estimates}
In all estimates made hence forth the constants $C, C_{1}, C_{2}, C_{3}, C_{\infty}$ are generic constants, that can change in value from line to line, and sometimes within the same line if so required.
Using positivity of the solution, and comparison arguments, one is easily able to prove global existence of classical solutions to system 
\eqref{eq:x1}-\eqref{eq:x3}.
We state the following result.

\begin{proposition}
\label{ge1}
 All positive solutions of the system
\eqref{eq:x1}-\eqref{eq:x3}, with initial data in $\mathbb{L}^{\infty }(\Omega )$ are classical and global.
\end{proposition}

See appendix \ref{app0} for the proof.
Although we have global existence of classical via proposition \ref{ge1}, we can actually derive uniform $L^{\infty}(\Omega)$ bounds, for initial data only in $L^2(\Omega)$. This will facilitate estimates required to show existence of bounded absorbing sets and the compactness of trajectories, which will lead to the existence of a global attractor in the product space $L^2(\Omega)$. Since the construction of global attractors require a Hilbert space setting, where the appropriate function spaces are the Hilbert spaces $L^{2}(\Omega)$ and $H^{s}(\Omega)$, $s=1,2$, we want to proceed by showing there is a weak solution in these function spaces. We then make estimates on this class of solutions, to proceed with the requisite analysis for the global attractors. We state the following theorem:
 
\begin{theorem}
\label{thm:csol}
Consider the three species food chain model described via \eqref{eq:x1}-\eqref{eq:x3}. For any initial data $(u_0,v_0,r_0)$ in  $L^{2}(\Omega)$, and spatial dimension $n=1, 2, 3$, there exists a global weak solution $(u,v,r)$ to the system, which becomes a strong solution and then a classical solution.
\end{theorem}

The proofs of proposition \ref{ge1} and theorem \ref{thm:csol} follow via remark \ref{ab} and the estimates referred to therein.

We state the following result next, 
\begin{lemma}
\label{lem:lemba}
Consider $(u,v,r)$ that are solutions to the diffusive three species food chain model described via \eqref{eq:x1}-\eqref{eq:x3}, for any $(u_{0},v_{0},r_{0}) \in L^{2}(\Omega)$, and spatial dimension $n=1, 2, 3$, there exists a time $t^{*}(||u_{0}||_{2},||v_{0}||_{2})$ , and a constant $C$ independent of time, initial data and the allee threshold $m$,
and dependent only on the other parameters in \eqref{eq:x1}-\eqref{eq:x3}, such that for any $t > t^{*}$ the following uniform estimates hold:

\begin{equation}
\label{eq:xnn1}
||u||^{2}_{2} \leq C, || v||^{2}_{2} \leq C, \int^{t+1}_{t}||\nabla u||^{2}_{2}ds \leq C, \int^{t+1}_{t}||\nabla v||^{2}_{2}ds \leq C, ||\nabla u||^{2}_{2} \leq C,  ||\nabla v||^{2}_{2} \leq C.
\end{equation}

\end{lemma}

This follows easily via the estimates and methods of \cite{P10,PK13,PK14}. 
%
%
%
Note what differs here from  \cite{P10,PK13,PK14} is the equation for $r$, particularly due to a Allee effect now being modeled. 

We now proceed with the estimate on $r$. We multiply \eqref{eq:x3} by $r$ and integrate by parts to obtain

\begin{equation}
\label{eq:x11}
\frac{1}{2}\frac{d}{dt}||r||^{2}_{2}  + d_{3}||\nabla r||^{2}_{2} + \frac{w_3}{||v||_{\infty}+D_{3}}||r||^{4}_{4} \leq \left(c+\frac{mw_{3}}{v+D_{3}} \right)||r||^{3}_{3}.
\end{equation}

Thus we obtain

\begin{equation}
\label{eq:x11n7}
\frac{1}{2}\frac{d}{dt}||r||^{2}_{2}  + d_{3}||\nabla r||^{2}_{2} + cm||r||^{2}_{2} + \frac{w_3}{||v||_{\infty}+D_{3}}||r||^{4}_{4} \leq \left(c+\frac{mw_{3}}{v+D_{3}} \right)||r||^{3}_{3}.
\end{equation}

We then use H\"{o}lder's inequality followed by Young's inequality to obtain

\begin{eqnarray}
\label{eq:x11r}
&&\frac{1}{2}\frac{d}{dt}||r||^{2}_{2}  + d_{3}||\nabla r||^{2}_{2} + cm||r||^{2}_{2} + \frac{w_3}{||v||_{\infty}+D_{3}}||r||^{4}_{4},  \nonumber \\
&& \leq \frac{w_3}{||v||_{\infty}+D_{3}}||r||^{4}_{4} + C_{1}(||v||_{\infty} + D_{3} + m^4)|\Omega|. \nonumber \\
\end{eqnarray}

Thus we obtain

\begin{equation}
\label{eq:x11n8}
\frac{1}{2}\frac{d}{dt}||r||^{2}_{2}  + cm||r||^{2}_{2}  \leq  C_{1}(||v||_{\infty} + D_{3} + m^4)|\Omega|,
\end{equation}

which implies

\begin{equation}
\label{eq:x11n2n}
||r||^{2}_{2} \leq e^{-cmt}||r_{0}||^{2}_{2} + \frac{ C_{1}(||v||_{\infty} + D_{3} + m^4)|\Omega|}{cm}.
\end{equation}

We now use the estimate for $||v||_{\infty}$ via \eqref{eq:liea}  (which does not depend on $r$ see remark \ref{rv}) to obtain 
the following estimates for $r$,

\begin{equation}
\label{eq:x11n2}
||r||^{2}_{2} \leq 1 + \frac{ C_{1}(C + D_{3} + m^4)|\Omega|}{cm},  \  \mbox{for} \ t > \max\left(t^{*}+1, \frac{ \ln(||r_{0}||^{2}_{2})}{  cm }\right).
\end{equation}

Integrating \eqref{eq:x11r} in the time interval $[t,t+1]$  we obtain
\begin{equation}
\label{eq:x11n23n}
 \int^{t+1}_{t}||\nabla r||^{2}_{2} ds \leq 1 + \frac{ C_{1}(C + D_{3} + m^4)|\Omega|}{cm} + 2|\Omega|^{\frac{1}{4}}\left(c+\frac{w_{4}}{D_{3}}\right) \left(\frac{w_{4}}{D_{3}}\right)^{\frac{1}{4}}, 
\end{equation}

for $t >  \max\left(t^{*}+1, \frac{ \ln(||r_{0}||^{2}_{2})}{  cm }\right)$.

Here, $C, C_{1}$ do not depend on $m$.

\begin{remark}
What we notice is that the estimate via \eqref{eq:x11n2}, \eqref{eq:x11n23n} depend singularly on $m$. That is they yield no information, if we try to pass to the limit as $m \rightarrow 0$
\end{remark}

We now estimate the gradient of $r$. See the appendix \ref{app2} for details. What the standard analysis in appendix \ref{app2} shows, is that we have an estimate of the form

\begin{equation}
\label{eq:f1ns}
 \mathop{\limsup}_{t \rightarrow \infty} ||\nabla r||^{2}_{2}  \leq C, 
\end{equation}

\begin{remark}

The constant $C$ in \eqref{eq:f1ns} is independent of time and initial conditions, \emph{but depend singularly on} $m$.
In our estimates we apply the uniform Gronwall lemma which requires us to use the estimate via \eqref{eq:x11n23n}. This also depends singularly on $m$, thus the estimate \eqref{eq:f1ns} is uniform with respect to time and initial data but depend singularly on $m$.
\end{remark}

\begin{remark}
\label{ab}
The uniform $H^{1}(\Omega)$ estimates via lemma \ref{lem:lemba} and \eqref{eq:f1d4nn} give us uniform $L^{6}(\Omega)$ bounds in $\mathbb{R}^{3}$. We can now prove the reaction terms are in $L^{p}(\Omega)$ for $p>\frac{3}{2}$. It suffices to  show that 

\begin{equation}
||u-u^{2}-w_{1}\frac{uv}{u+v}||_{2} \leq C||u||^{2}_{4} \leq C,
\end{equation}

\begin{equation}
||-a_{2}v+w_{2}\frac{uv}{u+v}-w_{3}\left(\frac{vr}{v+r}\right)||_{2} \leq C||v||_{2} \leq C,
\end{equation}

\begin{equation}
||r\big({{r-m}}\big)\bigg(c-\frac{w_4r} {v+D_3}\bigg)||_{2} \leq C||r||^{3}_{6} \leq C.
\end{equation}

But the above follows via the uniform $L^{6}(\Omega)$ bounds on $u,v,r$. This proves proposition \ref{ge1}. Furthermore the estimate via \eqref{eq:f1ns} and lemma \ref{lem:lemba} allow us to derive the appropriate $L^2(\Omega)$
and $H^{1}(\Omega)$ bounds on a Galerkin truncation of \eqref{eq:x1}-\eqref{eq:x3}, extract appropriate subsequebces and pass to the limit as is standard \cite{T97, SY02}, to prove theorem \ref{thm:csol}.
 
\end{remark}

\subsection{Uniform in the parameter $m$ $L^{2}(\Omega)$ and $H^{1}(\Omega)$ estimates for $r$}
Note the estimates via \eqref{eq:x11n2}, \eqref{eq:f1ns} are singular in $m$. This causes extensive difficulties if we try to pass to the limit as $m \rightarrow 0$, hence making it difficult to prove upper semicontinuity. Our next goal is to derive \emph{uniform in $m$ estimate} on $r$. These are derived in detail in the appendix \ref{app3}. However, the key estimate derived therein is

\begin{eqnarray}
\label{eq:naa}
||\nabla r||^{2}_{2}  \leq C_{K} e^{K C_{K}}, \ \mbox{for} \ t > t_{1}=t_{0} + 1.
\end{eqnarray}

Here $C_{K}, K$ are independent of $m$. Thus the $H^{1}(\Omega)$ estimate for $r$ can be made independent of $m$.

\subsection{Uniform $H^{2}(\Omega)$ estimates}

We will now estimate the $H^2(\Omega)$ norms of the solution. We rewrite \eqref{eq:x1} as
\[
u_t- d_1 \Delta u =  u-u^{2}-w_{1}\left(\frac{u v }{u+v}\right).
\]

\noindent We square both sides of the equation and integrate by parts over $\Omega$ to obtain
\begin{eqnarray}
\label{eq:x1pq}
&&||u_t||^{2}_{2} + d_{1}||\Delta u||^{2}_{2}  + \frac{d}{dt}||\nabla u||^{2}_{2} , \nonumber \\
&=& \left\| \left( u-u^{2}-w_{1}\left(\frac{u v }{u+v}\right) \right) \right\|_2^2  ,\nonumber \\
&\leq& C\left(||u||^{2}_{2}+(w_{1})^2||u||^{2}_{2}  + ||u||^{4}_{4}\right) \leq C. \nonumber
\end{eqnarray}
\noindent This result follows by the embedding of $H^{1}(\Omega) \hookrightarrow L^4(\Omega) \hookrightarrow L^2(\Omega)$.  Therefore, we obtain
\[
||u_t||^{2}_{2} + d_{1}||\Delta u||^{2}_{2}  + \frac{d}{dt}||\nabla u||^{2}_{2}  \leq C\left(||u||^{2}_{2}+(w_{1})^2||u||^{2}_{2}  + ||u||^{4}_{4}\right) \leq C.
\]

We now make uniform in time estimates of the higher order terms. Integrating the estimates of \eqref{eq:x1pq} in the time interval $[T,T+1]$, for $T > t^{*}$%
\[
\int^{T+1}_{T}||u_t||^{2}_{2} dt \leq C_1, \   \int^{T+1}_{T}||\Delta u||^{2}_{2}dt  \leq C_2.
\]
\noindent Similarly we adopt the same procedure for $v$ to obtain
\[
\int^{T+1}_{T}||v_t||^{2}_{2}dt \leq C_1, \   \int^{T+1}_{T}||\Delta v||^{2}_{2}dt \leq C_2.
\]

Next, consider the gradient of \eqref{eq:x1}. Following the same technique as in deriving \eqref{eq:x1pq} we obtain for the left hand side
\begin{eqnarray}
\label{eq:gut1}
&& ||\nabla u_t||^{2}_{2} + d_{1}||\nabla (\Delta u)||^{2}_{2}  + \frac{d}{dt}||\Delta u||^{2}_{2} + \int_{\partial \Omega}\Delta u \nabla u_{t} \cdot \textbf{n} dS, \nonumber \\
&&=||\nabla u_t||^{2}_{2} + d_{1}||\nabla (\Delta u)||^{2}_{2}  + \frac{d}{dt}||\Delta u||^{2}_{2} + \int_{\partial \Omega}\Delta u \frac{\partial}{\partial t}(\nabla u \cdot \textbf{n} )dS, \nonumber \\
&&= ||\nabla u_t||^{2}_{2} + d_{1}||\nabla (\Delta u)||^{2}_{2}  + \frac{d}{dt}||\Delta u||^{2}_{2}.
\end{eqnarray}
\noindent This follows via the boundary condition. Thus we have
\begin{eqnarray}
\label{eq:x1pq1}
&& ||\nabla u_t||^{2}_{2} + d_{1}||\nabla (\Delta u)||^{2}_{2}  + \frac{d}{dt}||\Delta u||^{2}_{2},  \nonumber \\
&& = \left(\nabla ( u-u^{2}-w_{1}\left(\frac{u v }{u+v}\right) )\right)^2, \nonumber \\
&& \leq C(|| u||^{4}_{4}  + || v||^{4}_{4}  + || \Delta u||^{4}_{2} + || \nabla v||^{2}_{2} +  || \nabla u||^{2}_{2}).
\end{eqnarray}
\noindent This follows via Young's inequality with epsilon, as well as the embedding of $H^{2}(\Omega)\hookrightarrow W^{1,4}(\Omega) \hookrightarrow H^{1}(\Omega)$. This implies that
\[
\frac{d}{dt}||\Delta u||^{2}_{2}  \leq C || \Delta u||^{4}_{2} + C(|| u||^{4}_{4}  + || v||^{4}_{4}  + || \nabla v||^{2}_{2} +  || \nabla u||^{2}_{2}).
\]
\noindent We now use the uniform Gronwall lemma with
\[ \beta(t) = ||\Delta u||^{2}_{2}, \ \zeta(t)=||\Delta u||^{2}_{2} , \ h(t)= C(|| u||^{4}_{4}  + || v||^{4}_{4}  + || \nabla v||^{2}_{2} +  || \nabla u||^{2}_{2}),~~ q=1, \]
\noindent to obtain
\[
||\Delta u||_{2} \leq C,  \ ||\Delta v||_{2} \leq C, \ \mbox{for} \ t  > t^{*} + 1.
\]
The estimate for $v$ is derived similarly.
\begin{remark}
\label{rv}
It is critical to notice that the estimate for $||\Delta v||_{2}$ does not involve $r$. This is because in order to derive it, we adopt same procedure as in \eqref{eq:x1pq}, instead for the $v$ equation, and square both sides to proceed. Therein notice that the term involving $r$, $\frac{vr}{v+r}$ can be estimated as $||\frac{vr}{v+r}||_{2} \leq C||\frac{r}{v+r}||_{\infty} ||v||_{2} \leq C_{1} ||v||_{2}$. So the right hand side has no dependence on $r$.
\end{remark}

In order to estimate the $H^2(\Omega)$ norm of $r$ we proceed similarly to obtain

\begin{eqnarray}
\label{eq:x1pqr}
||r_t||^{2}_{2} + d_{3}||\Delta r||^{2}_{2}  + \frac{d}{dt}||\nabla r||^{2}_{2}&=& \left\|  \left( \left( c+\frac{mw_{4}}{v+D_{3}}\right)r^{2} - \frac{w_{4}}{v+D_{3}}r^{3} - mcr  \right) \right\|_2^2,  \nonumber \\
&\leq& C\left(||r||^{2}_{2}+||r||^{4}_{4}\right)  + ||r||^{6}_{6}, \nonumber \\
&\leq & C \left(||\nabla r||^{2}_{2}\right)^{3}. \nonumber \\
\end{eqnarray}

We now make uniform in time estimates of the higher order terms. Integrating the estimates of \eqref{eq:x1pqr} in the time interval $[T,T+1]$, for $T > t_{1}$, and using the estimates via \eqref{eq:nrn} yields
\[
\int^{T+1}_{T}||r_t||^{2}_{2} dt \leq C_1, \  \int^{T+1}_{T}||\Delta r||^{2}_{2}dt \leq C_2.
\]

Next, consider the gradient of \eqref{eq:x3}. Following the same technique as in deriving \eqref{eq:x1pq} we obtain 


\begin{eqnarray}
\label{eq:x1pqr4}
 ||\nabla r_t||^{2}_{2} + d_{3}||\nabla (\Delta r)||^{2}_{2}  + \frac{d}{dt}||\Delta r||^{2}_{2}&& = \left(\nabla \left( \left( c+\frac{mw_{4}}{v+D_{3}}\right)r^{2} - \frac{w_{4}}{v+D_{3}}r^{3} - mcr  \right) \right)^2, \nonumber \\
&& \leq C\left( ||\nabla r||^{2}_{2}\right)^{2}. \nonumber \\
\end{eqnarray}

This yields 

\begin{equation}
\label{eq:lieah2}
 \frac{d}{dt}||\Delta r||^{2}_{2}  \leq C\left( ||\nabla r||^{2}_{2}\right)^{2} \left( ||\Delta r||^{2}_{2}\right)+ C_{2}\left( || r||^{2}_{2}\right)^{2}.
\end{equation}

We can now use the estimates via \eqref{eq:rl2n}, \eqref{eq:nrn} in conjunction with the uniform Gronwall lemma to yield a uniform bound on $||\Delta r||_{2}$.
\noindent Thus, via elliptic regularity,
\begin{equation}
\label{eq:h2ea}
|| u||_{H^{2}(\Omega)} \leq C,  \ ||v||_{H^{2}(\Omega)} \leq C, \  ||r||_{H^{2}(\Omega)}  \ \leq C \  \mbox{for} \ t > \max(t_{1}+1,t^{*} + 1).
\end{equation}
\noindent Since $H^2(\Omega) \hookrightarrow L^{\infty}(\Omega)$ in $\mathbb{R}^2$ and $\mathbb{R}^{3}$, the following estimate is valid in $\mathbb{R}^2$ and $\mathbb{R}^{3}$
\begin{equation}
\label{eq:liea}
||u||_{\infty} \leq C,  \ ||v||_{\infty} \leq C, \ ||r||_{\infty} \leq C  \ \mbox{for} \ t >  \max(t_{1}+1,t^{*} + 1).
\end{equation}

\section{Long Time Dynamics}
\subsection{Existence of global attractor}
In this section we prove the existence of a $(H,H)$ global attractor for system \eqref{eq:x1}-\eqref{eq:x3}, which is subsequently demonstrated to be a $(H,E)$ attractor. 
We now state the following theorem,

\begin{theorem}
\label{thm:ga1}
Consider the reaction diffusion equation described via \eqref{eq:x1}-\eqref{eq:x3} where $\Omega$ is of spatial dimension $n=1, 2, 3$. There exists a $(H,H)$ global attractor $\mathcal{A}$ for the system. This is compact and invariant in $H$, and it attracts all bounded subsets of $H$ in the $H$ metric.
\end{theorem}

\textbf{proof}
We have shown that the system is well posed via theorems \ref{thm:csol}. Thus there exists a well defined semigroup $\left\{S(t)\right\}_{t \geq 0}:H \rightarrow H$. The estimates derived in Lemma \ref{lem:lemba}  demonstrate the existence of bounded absorbing sets in $H$ and $E$. Thus given a sequence $\left\{u_{0,n}\right\}^{\infty}_{n=1}$, that is bounded in $L^{2}(\Omega)$,  we know that for $t > t^{*}$, 
 \begin{equation}
 S(t)(u_{0,n}) \subset B \subset H^{1}(\Omega).
 \end{equation}
 Here $B$ is the bounded absorbing set in $H^{1}(\Omega)$ from lemma \ref{lem:lemba}. Now for n large enough $t_{n} > t^{*}$, thus for such $t_{n}$ we have
 
 \begin{equation}
 S(t_{n})(u_{0,n}) \subset B \subset H^{1}(\Omega).
 \end{equation}
 
 This implies that we have the following uniform bound, 
 
 \begin{equation}
 \label{eq:h1ga}
 ||S(t_{n})(u_{0,n})||_{H^{1}(\Omega)} \leq C ,
 \end{equation}
 
This implies via standard functional analysis theory, see \cite{T97}, the existence of a subsequence still labeled $S(t_{n})(u_{0,n})$ such that

 \begin{equation}
 S(t_{n})(u_{0,n}) \rightharpoonup u \ \mbox{in} \ H^{1}(\Omega),
 \end{equation}
 
 Which implies via the compact Sobolev embedding of 
 
\begin{equation}
E \hookrightarrow H,
\end{equation}

that

 \begin{equation}
 S(t_{n})(u_{0,n}) \rightarrow u \ \mbox{in} \ L^{2}(\Omega).
 \end{equation}
 This yields the asymptotic compactness of the semigroup $\left\{S(t)\right\}_{t \geq 0}$ in $H$. The convergences for the $v,r$ components follow similarly. The theorem is now proved.
$\square$

Next we can state the following theorem 

\begin{theorem}
\label{thm:ga2}
Consider the reaction diffusion equation described via \eqref{eq:x1}-\eqref{eq:x3} where $\Omega$ is of spatial dimension $n=1, 2, 3$. There exists a $(H,E)$ global attractor $\mathcal{A}$ for the system. This is compact and invariant in $H$, and it attracts all bounded subsets of $H$ in the $E$ metric.
\end{theorem}
 
\textbf{proof}
The proof essentially follows verbatim theorem \ref{thm:ga1}. The existence of a bounded absorbing set in $E$ follows from lemma \ref{lem:lemba}. In order to prove the asymptotic compactness in $E$ we use the uniform $H^2(\Omega)$ estimates in \eqref{eq:h2ea}, and the compact Sobolev embedding of $H^2(\Omega) \hookrightarrow H^1(\Omega)$.
$\square$

\subsection{Upper-semicontinuity of Global attractor with respect to the Allee parameter $m$}
\label{alleeThreshold}

Here we state and prove the main semicontinuity result. This result follows via a series of estimates, and theorems, all of which we derive for the reader in appendix \ref{app4}

\begin{theorem}
\label{thm:gaus1sc}
Consider the reaction diffusion equation described via \eqref{eq:x1}-\eqref{eq:x3} where $\Omega$ is of spatial dimension $n=1, 2, 3$. Given a set of positive parameters excluding $m$, the family of global attractors is upper semicontinuous with respect to the Allee threshold $m \geq 0$ as it converges to zero. That is

\begin{equation}
dist_{E}(\mathcal{A}_{m},\mathcal{A}_{0}) \rightarrow 0, \ \mbox{as} \ m \rightarrow 0^{+}.
\end{equation}
\end{theorem}
\begin{proof}
We know that for $\epsilon > 0$, the global attractor $\mathcal{A}_{0}$ for the reaction diffusion equation described via \eqref{eq:x1}-\eqref{eq:x3} attracts $\mathcal{U}$. This follows via theorems \ref{thm:gaus1}, \ref{thm:gaus12} and  \ref{thm:gaus13}. Thus there exists a finite time $t_{\epsilon} > 0$, such that

\begin{equation}
S_{0}(t_{\epsilon})\mathcal{U} \subset \mathcal{N}_{E}(\mathcal{A}_{0},\frac{\epsilon}{2}).
\end{equation}

Here, $\mathcal{N}_{E}(\mathcal{A}_{0},\frac{\epsilon}{2})$ is the $\frac{\epsilon}{2}$ ball of $\mathcal{A}_{0}$ in the space $E$.
Now we also know by the uniform convergence theorem, theorem \ref{thm:gaum} that there exists a $m_{\epsilon} \in (0,1]$ such that

\begin{equation}
\sup_{g_{0} \in \mathcal{U}} ||S_{m}(t_{\epsilon})g_{0} - S_{0}(t_{\epsilon})g_{0}||_{E} \leq \frac{\epsilon}{2} , \ \mbox{for} \ \mbox{any} \ m \in m_{\epsilon}.
\end{equation}

Since we know $\mathcal{A}_{m}$ is invariant we have

\begin{equation}
\mathcal{A}_{m} = S_{m}(t_{\epsilon})\mathcal{A}_{m} \subset S_{m}(t_{\epsilon})\mathcal{U} \subset \mathcal{N}_{E}(S_{0}(t_{\epsilon})\mathcal{U},\frac{\epsilon}{2})\subset \mathcal{N}_{E}(\mathcal{A}_{0},\epsilon).
\end{equation}

Thus one obtains the upper semicontinuity,

\begin{equation}
dist_{E}(\mathcal{A}_{m},\mathcal{A}_{0}) \rightarrow 0, \ \mbox{as} \ m \rightarrow 0^{+}.
\end{equation}
 This proves the theorem.

\end{proof}

\subsection{Computing the decay rate with respect to the Allee threshold parameter $m$}

In the previous subsection \ref{alleeThreshold}, the role of the Allee threshold parameter $m$ on the long time dynamics of system \eqref{eq:x1}-\eqref{eq:x3} is explored. Theorem \ref{thm:gaus1sc}
shows that the global attractors $\mathcal{A}_{m}$ converge to $\mathcal{A}_{0}$ in the $H^{1}(\Omega)$ norm, as $m$ decreases to zero. It is also seen how changing $m$ effects the Turing dynamics in subsection \ref{turing}.
Theorem \ref{thm:gaus1sc} is a purely theoretical result. It gives no information as to the exact decay rate in terms of the parameter $m$. In general, there are no results in the literature to estimate this decay rate. We now aim to examine the role of
$m$ with respect to convergence, as $m$ decreases to zero as discussed
earlier. We aim to explicitly find the decay rate of $\mathcal{A}_{m}$ to a target attractor $\mathcal{A}_{0}$ (that is when $m=0$). We choose Turing patterns as the state of interest. In particular the numerical scheme discussed in \ref{turing} was employed, and the Turing patterns at an extended time are examined. The goal is to
gather numerical evidence as to the dependence on $m$ with respect to
the convergence rate,
\begin{eqnarray}
  \label{eq:mBounds}
  \| U_m - U_0 \|_{H^1(\Omega)} & < & C m^\gamma,
\end{eqnarray}
where $U_0=(u_, v_0, r_0)$ is the solution for $m=0$, $U_m=(u_m,v_m,r_m)$ is the solution for any
given value of $m$, and both $C$ and $\gamma$ are constants.

The numerical approximation for the case $m=0$ was established at the
scaled time, $t=10,000$. The approximations, $u_m(x,t)$, $v_m(x,t)$,
and $r_m(x,t)$, were then established for a range of values of $m$
with 121 equally spaced values from 0 to 0.0035. The approximation for
each value of $m$ was established up to $t=10,000$, and it was
determined that for each value of $m$ in this range a fixed spatial
pattern was found. The $H^1(\Omega)$ errors for the numerical approximation
were then calculated with respect to the $m=0$ case,
\begin{eqnarray*}
  \| u_m(\cdot,10,000) - u_0(\cdot,10,000)\|_{H^1(\Omega)}, \\
  \| v_m(\cdot,10,000) - v_0(\cdot,10,000)\|_{H^1(\Omega)}, \\
  \| r_m(\cdot,10,000) - r_0(\cdot,10,000)\|_{H^1(\Omega)}.
\end{eqnarray*}

\begin{figure}[!htp]
	\begin{center}
	\includegraphics[scale=0.4]{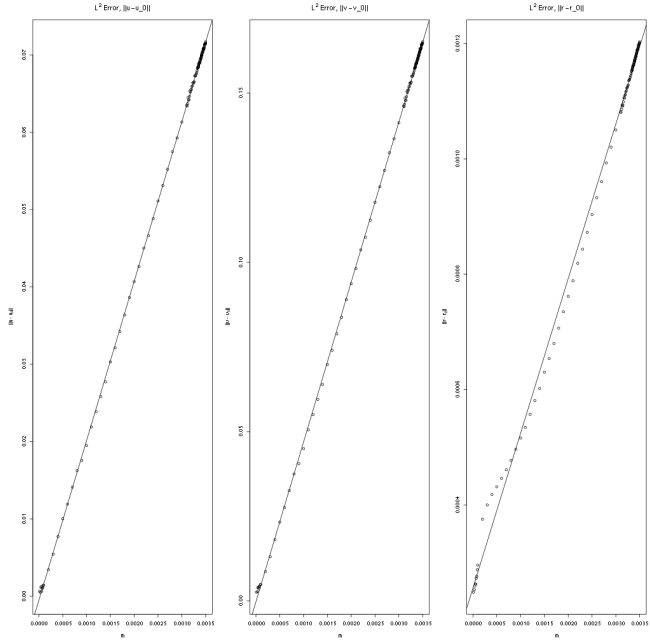}
	\end{center}
    \caption{The $H^1(\Omega)$ errors for the differences
      $\|u_m(\cdot,10,000)-u_0(\cdot,10,000)\|$,
      $\|v_m(\cdot,10,000)-v_0(\cdot,10,000)\|$, and
      $\|r_m(\cdot,10,000)-r_0(\cdot,10,000)\|$ are shown for a range
      of values of $m$. The solid line is the linear least squares
      regression line.  }
	\label{fig:l2Error}
\end{figure}

\begin{figure}[!htp]
	\begin{center}
	\includegraphics[scale=0.4]{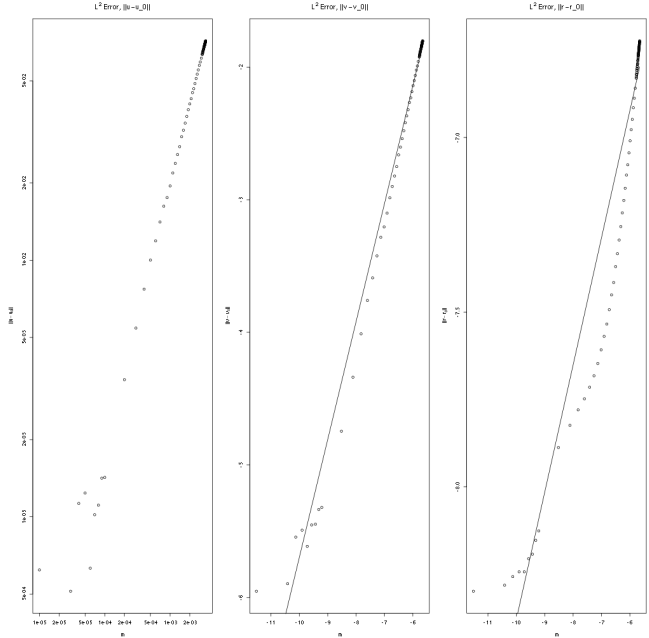}
	\end{center}
    \caption{The $H^1(\Omega)$ errors for the differences
      $\|u_m(\cdot,10,000)-u_0(\cdot,10,000)\|$,
      $\|v_m(\cdot,10,000)-v_0(\cdot,10,000)\|$, and
      $\|r_m(\cdot,10,000)-r_0(\cdot,10,000)\|$ are shown for a range
      of values of $m$ but are plotted in a log log format. The solid
      line is the linear least squares regression line for the log-log
      data.  }
	\label{fig:l2ErrorMLog}
\end{figure}

The errors are shown in Figure \ref{fig:l2Error} and Figure
\ref{fig:l2ErrorMLog}. The first figure, Figure \ref{fig:l2Error},
shows the raw errors, and the second figure, Figure
\ref{fig:l2ErrorMLog}, shows the errors as a log-log plot. A linear
least squares fit for the data was also calculated, and the best fit
straight line is shown on both plots.  In both cases the close linear
fit for larger values of $m$ contribute to strong, positive
correlations in the data. However, for smaller values $m$ the log-log
data moves away from a straight line. For the smaller values of $m$
the errors in the numerical approximation are of the same order as the
error in the $H^1(\Omega)$ norm, and the data no longer follows the same
trend.

Examining the errors for $u_m$, the correlation for the raw data is
0.99995, and the data reveals a strong, positive linear relationship.
In this scenario, it is assumed that the value of $\gamma$ in equation
\eqref{eq:mBounds} is one. The approximation for the slope of the raw
data gives an estimate of 20.6 with a 95\% confidence interval between
20.6215 and 20.6206. The 95\% confidence interval for the intercept is
between -5.1E-5 and -9.0E-4. Due to the large number of samples, the
error in the confidence interval is very small, but the estimate for
the parameter is close to zero.

Examining the errors for the log of the errors correlation is 0.993
for the log-log data. The approximation for the slope of the log-log
data gives an estimate of 1.003 with a 95\% confidence interval
between 0.72 and 1.28. The 95\% confidence interval for the intercept
is between 2.7 and 3.3. Note that the value of one is within the
confidence interval making it difficult to claim that the relationship
is not linear. The estimate for the slope, in fact, is very close to
one indicating that the value of $\gamma$ in equation
\eqref{eq:mBounds} is close to one.

\section{Overexploitation Phenomenon}

Overexploitation refers to the phenomenon where excessive harvesting of a species can result in its extinction. This finds applications in conservation, biodiversity, cascade control and so on. Although there is a fair amount of literature on this in two species models \cite{shi11}, it is much less studied in three species models. 
To begin, we state the following theorem

\begin{theorem}[Middle predator and Allee effect mediated overexploitation]
\label{thm:ox}
Consider $(u,v,r)$ that are solutions to the diffusive three species food chain model described via \eqref{eq:x1}-\eqref{eq:x3}. If $w_{1} > a_{2}+1+w_{3}$, then for any given initial prey density $u_{0}$ there exists a threshold $M_{1}=\frac{1}{\alpha}||u_{0}||_{p}$, s.t if $M_{1} <||v_{0}||_{p}$ and $||r_{0}||_{\infty} < \min(m,\frac{cD_{3}}{w_{4}})$, then $(u,v,r) \rightarrow (0,0,0)$ uniformly for $x \in \bar{\Omega}$ as $t \rightarrow \infty$.
\end{theorem}

\textbf{proof:}
We begin by looking at the equation for $u$, and following the ideas in \cite{K98}

\begin{equation}
\label{eq:x1x0}
\frac{\partial u}{\partial t}= d_1\Delta u + u-u^{2}-w_{1}\frac{uv}{u+v} = d_1\Delta u + u-u^{2}-w_{1}\frac{u}{\frac{u}{v}+1}.
\end{equation}

We will claim that for all $t$, $||\frac{u}{v}||_{\infty} < \alpha$, and $\lim_{t \rightarrow \infty}||u||_{p} \rightarrow 0$, $\forall p$, if $\frac{||u_{0}||_{p}}{||v_{0}||_{p}} < \alpha$. If not there exists a first time $t_{1}$ s.t.,
$\frac{||u||_{p}}{||v||_{p}} < \alpha$ on $t \in [0,t_{1}]$ and $\frac{||u(t_{1})||_{p}}{||v(t_{1})||_{p}} = \alpha$.

Due to $w_{1} > a_{2}+1 + w_{3}$, there exists an $\alpha$ s.t $\frac{w_{1}}{1+\alpha} = 1 + a_{2}+w_{3}$, thus 

\begin{eqnarray}
\label{eq:x1x0}
 \frac{\partial u}{\partial t}= d_1\Delta u + u-u^{2}-w_{1}\frac{u}{\frac{u}{v}+1} &&  \leq  d_1\Delta u + u-w_{1}\frac{u}{\alpha+1}, \nonumber \\
&& = d_1\Delta u + u-u -a_{2}u -w_{3}u, \nonumber \\
&& = d_1\Delta u -(a_{2}+w_{3})u.
\end{eqnarray}

Thus $u$ satisfies

\begin{equation}
||u||_{p} \leq ||u_{0}||_{p} e^{-(a_{2}+w_{3})t}.
\end{equation}

This easily follows via multiplying \eqref{eq:x1} by $|u|^{p-1}$, and integrating by parts.
We also notice that the $\tilde{v}$ solving the following equation, with $\tilde{v}_{0} = \sup v_{0}(x)$, is a super solution to \eqref{eq:x2}

\begin{equation}
\label{eq:x2xo}
\frac{\partial \tilde{v}}{ \partial t} =  -(a_{2}+w_{3})\tilde{v}.
\end{equation}

Multiplying through by $|v|^{p-1}$, and integrating by parts yields

\begin{equation}
\label{eq:x2xo1}
||v||_{p} \geq ||v_{0}||_{p} e^{-(a_{2}+w_{3})t}.
\end{equation}

This implies

\begin{equation}
\label{eq:x2xo1}
\frac{||u||_{p}}{||v||_{p}}  \leq \frac{||u_{0}||_{p}}{||v_{0}||_{p}}   < \alpha. 
\end{equation}

This implies a contradiction. Thus $\lim_{t \rightarrow \infty}||u||_{p} \rightarrow 0$, which implies uniform convergence of a subsequence, say $u_{n_{j}}$ to 0 (where $u_{n}$ is a Galerkin truncation of $u$ and standard theory is adopted \cite{SY02}), and by uniqueness of solutions,
$u \rightarrow 0$ uniformly, as $t \rightarrow \infty$. This reduces the $v$ equation to

\begin{equation}
\label{eq:x2ne}
\frac{\partial v}{ \partial t}= d_2 \Delta v - a_{2}v -w_{3}\left(\frac{vr}{v+r}\right) \leq d_2 \Delta v - a_{2}v,
\end{equation}

and it is easy to see that $v \rightarrow 0$ uniformly, as $t \rightarrow \infty$. 

Now the $r$ equation reduces to 

\begin{equation}
\frac{\partial r}{ \partial t} = d_3 \Delta r + r\big({{r-m}}\big)\bigg(c-\frac{w_4r} {D_3}\bigg),
\end{equation}

and if $||r_{0}||_{\infty} < \min(m,\frac{cD_{3}}{w_{4}})$ standard analysis yields 

$r \rightarrow 0$ uniformly. 

This proves the theorem.

$\square$

\begin{remark}
Note, in the three species food chain model \eqref{eq:x1}-\eqref{eq:x3},  overexploitation has to be middle predator mediated, and also depends on the allele threshold.
Clearly if  $||r_{0}||_{p} > \max(m,\frac{cD_{3}}{w_{4}})$, or $\frac{cD_{3}}{w_{4}} < ||r_{0}||_{p}  < m$, or $\frac{cD_{3}}{w_{4}} > ||r_{0}||_{p}  > m$, $r$ cannot go extinct. 
\end{remark}

Also, if the attack rate of the top predator $r$ is large enough, it could cause $v$ to go extinct, before $u$ does. Once $v$ goes extinct $u \rightarrow 1$, again prohibiting overexploitation. 

\begin{theorem}[Top predator mediated overexploitation avoidance]
\label{thm:ox1}
Consider $(u,v,r)$ that are solutions to the diffusive three species food chain model described via \eqref{eq:x1}-\eqref{eq:x3}. If $w_{3} > w_{2}+1+w_{1}$, then for any given initial prey density $u_{0}$ there exists a threshold $M_{2}=\frac{1}{\alpha_{1}}||r_{0}||_{p}$, s.t if $\max \left(M_{2},m, \frac{cD_{3}}{w_{4}}\right) < ||r_{0}||_{p}$ then $(u,v,r) \rightarrow (1,0,r^{*})$ uniformly as $t \rightarrow \infty$. 
\end{theorem}

\textbf{proof}:
We first derive a lower bound on $u$. Trivially we have $||u||_{\infty} \leq 1$. Now 

\begin{equation}
\frac{\partial u}{ \partial t} \geq  - u^2 -w_{1}\frac{vu}{v+u} \geq -u-w_{1}u.
\end{equation}

This yields 

\begin{equation}
||u||_{p} \geq ||u_{0}||_{p} e^{-(1+w_{1})t}.
\end{equation}

Next we will claim that for all $t$, $||\frac{v}{r}||_{\infty} < \alpha_{1}$, and $\lim_{t \rightarrow \infty}||v||_{p} \rightarrow 0$, $\forall p$, if $\frac{||v_{0}||_{p}}{||r_{0}||_{p}} < \alpha_{1}$. If not there exists a first time $t_{2}$ s.t.,
$\frac{||v||_{p}}{||r||_{p}} < \alpha_{1}$ on $t \in [0,t_{2}]$ and $\frac{||u(t_{2})||_{p}}{||v(t_{2})||_{p}} = \alpha_{1}$.

Due to the choice of $w_{3} > w_{2}+1 + w_{1}$, there exists an $\alpha_{1}$ s.t $\frac{w_{3}}{1+\alpha_{1}} = 1 + w_{2}+w_{1}$, thus 

\begin{eqnarray}
\label{eq:x1x0}
  \frac{\partial v}{\partial t} && = d_2\Delta v  - a_{2}v + w_{2}\frac{vu}{v+u} - w_{3}\frac{v}{\frac{v}{r}+1},  \nonumber \\
&& \leq d_2\Delta v  - a_{2}v + w_{2}v - \frac{w_{3}}{1+\alpha_{1}}v, \nonumber \\
&& = d_2\Delta v  - a_{2}v + w_{2}v - (1 + w_{2}+w_{1})v. \nonumber \\
\end{eqnarray}

Thus $v$ satisfies

\begin{equation}
||v||_{p} \leq ||v_{0}||_{p} e^{-(a_{2}+1+w_{3})t}.
\end{equation}

By the restriction on $r_{0}$, we see that trivially $r \rightarrow m$ or $r \rightarrow \frac{cD_{3}}{w_{3}}$.
This implies

\begin{equation}
\label{eq:x2xo1}
\frac{||v||_{p}}{||r||_{p}}  \leq \frac{||v_{0}||_{p}}{||r_{0}||_{p}}   < \alpha_{1}.
\end{equation}

This implies a contradiction. Thus $\lim_{t \rightarrow \infty}||v||_{p} \rightarrow 0$, which implies uniform convergence to 0 of a subsequence, say $v_{n_{j}}$ (where $v_{n}$ is a Galerkin truncation of $v$ and standard theory is adopted \cite{SY02}), and by uniqueness $v \rightarrow 0$ uniformly as $t \rightarrow \infty$. Looking at the decay rate for $v$, we see it converges to 0 faster than $u$, so $u > 0$, when $v \rightarrow 0$ uniformly, and so then the $u$ equation reduces to

\begin{equation}
\frac{\partial u}{ \partial t}= d_1 \Delta u + u - u^2,
\end{equation}

and we know via standard theory that now
 $u \rightarrow 1$ uniformly. 

This proves the theorem.
$\square$

\begin{remark}
This also happens if $\frac{cD_{3}}{w_{4}} < ||r_{0}||_{\infty} < m$, or $\frac{cD_{3}}{w_{4}} > ||r_{0}||_{\infty} > m$.
\end{remark}

The following lemma describes persistence in the top predator.
\begin{lemma}[Persistence in top predator]
\label{lem:ox2}
Consider $(u,v,r)$ that are solutions to the diffusive three species food chain model described via \eqref{eq:x1}-\eqref{eq:x3}. If $\min(m,\frac{cD_{3}}{w_{4}}) < ||r_{0}||_{\infty}$, then $r \rightarrow \max(m,\frac{cD_{3}}{w_{4}})$ uniformly as $t \rightarrow \infty$, hence persists.
\end{lemma}

\textbf{proof:}
Note $c- \frac{w_{4}r}{v+D_{3}} >  c- \frac{w_{4}r}{D_{3}} $, due to positivity of $r,v$. Thus
we can compare the solution of \eqref{eq:x3} to the solution of the ODE

\begin{equation}
\label{eq:x2nn}
\frac{d\tilde{r}}{dt} = \tilde{r}(\tilde{r}-m)\bigg(c- \frac{w_{4}\tilde{r}}{v+D_{3}}\bigg),
\end{equation}

with $\tilde{r}_{0}= \min(r_{0}(x)) > \min(m,\frac{cD_{3}}{w_{4}}) $. Clearly $\tilde{r} \rightarrow \max(m,\frac{cD_{3}}{w_{4}})$ as $t \rightarrow \infty$, and since $r$ solving \eqref{eq:x3} is a supersolution to $\tilde{r}$ solving \eqref{eq:x2nn}, we have $\liminf_{t \rightarrow \infty} \min_{\bar{\Omega}} (r(x,t) \geq \max(m,\frac{cD_{3}}{w_{4}})$.

$\square$ 
\begin{remark}
In the event that there is no Allee effect, or $m=0$, $r \rightarrow \frac{cD_{3}}{w_{4}}$ uniformly as $t \rightarrow \infty$, hence always persists, independent of all other parameters or initial data $(u_{0}, v_{0}, r_{0})$. Thus in a three species system where the top predator is a generalist, and can switch its favorite food source, true overexploitation can only take place if an Allee effect is in place, and not otherwise.
\end{remark}

\begin{remark}
All of the over exploitation theorems were verified numerically, for various ranges of parameter values.
\end{remark}

\section{Turing Instability}
\label{turing}
In this section we establish the condition needed to ensure a positive interior steady state for model system (\ref{eq:x1}-\ref{eq:x3}). We focus on deriving the condition necessary and sufficient for Turing instability to occur as a result of the introduction of diffusion. This phenomena is referred to as \emph{diffusion driven instability}, which was first introduced by Alan Turing\cite{Turing52}. We first establish the steady state for model system (\ref{eq:x1od}-\ref{eq:x3od})
\subsection{Steady States}
Before we can begin the analysis of Turing instability, we have to establish the conditions needed to ensure a positive interior steady state for the ODE version of model system \eqref{eq:x1}-\eqref{eq:x3}, which is given as 

\begin{align}
&\frac{du}{dt}=  u-u^{2}-w_{1}\frac{uv}{u+v},\label{eq:x1od} \\ 
&\frac{dv}{ dt}=  -a_{2}v+w_{2}\frac{uv}{u+v}-w_{3}\left(\frac{vr}{v+r}\right),\label{eq:x2od}\\ 
&\frac{dr}{ dt} =  r\big({{r-m}}\big)\bigg(c-\frac{w_4r} {v+D_3}\bigg). \label{eq:x3od}
\end{align}

Also all parameters associated with model system (\ref{eq:x1od}-\ref{eq:x3od}) are assumed to be positive constants and have the usual biological meaning. \\

The model system \eqref{eq:x1od}-\eqref{eq:x3od} has the following non-negative steady states 
\begin{enumerate}
\item Total Extinction of the three species: $E_0(0,0,0)$.
\item Predator free steady state: $E_1(1,0,0)$.
\item No primary food source for predator steady state: $E_2(0,0,m)$ and $E_3(0,0,\frac{cD_3} {w_4})$.
\item Top Predator free steady state: $E_4(\tilde {u},\tilde {v},0)$ where $\tilde {u}=\frac{w_2-w_1w_2+a_2w_1} {w_2},\tilde {v}=\frac{(w_2-a_2)\tilde{u}}{a_2}$ if $w_2>w_1(w_2-a_2),w_2>a_2$.
\item Middle Predator free steady state: $E_5(1,0,\frac{cD_3} {w_4})$ and $E_6(1,0,m)$.
\item Coexistence of Species: $E_7(\tilde {\tilde {u}},\tilde {\tilde {v}},m)$ where $\tilde {\tilde {v}}= \frac{(1-\tilde {\tilde {u}})\tilde {\tilde {u}}}{(w_1+\tilde {\tilde {u}}-1)}, (1-{w_1})<\tilde {\tilde {u}} <1, w_1<1$ and $\tilde {\tilde {u}}$ is the real positive root of following equation:\\ 
\begin{equation}\nu_1u^3+\nu_2u^2+\nu_3u+\nu_4=0,
\label{eq:4}
\end{equation} where $ \nu_1=w_2,\, \nu_2 = -[a_2w_1+w_2(-w_1+(2+m))], \nu_3= [a_2w_1-w_1w_2+w_2+mw_1w_3-m(-a_2w_1+2w_2(w_1-1))], \, \nu_4=m(w_1-1)[w_2+w_1(w_3-(w_2-a_2))].$ 
\item Coexistence of Species with interior equilibrium population: $E_8(u^{*},v^{*},r^{*})$ is the solution of following equations:
\begin{subequations} \label{eq:5}
\begin{eqnarray}
1-u-\frac{w_1v} {u+v}=0, \\
 -a_2+\frac{w_2u} {u+ v}-\frac{w_3r} { r+v}=0,\label{eq:5ab}\\
 c-\frac{w_4 r} {v+D_3}=0.\label{eq:5ac}
\end{eqnarray}
\end{subequations}
\end{enumerate}
From \eqref{eq:5} we have,\\
\begin{eqnarray}
\label{eq:6}
v=\frac{(1-u)u} {w_1+u-1},\\
\label{eq:1.11}
r=\frac{c(v+D_3)} {w_4}.
\end{eqnarray}
Putting the values of $v$ and $r$ from \eqref{eq:6} and \eqref{eq:1.11} into \eqref{eq:5ab}, we get
\begin{eqnarray}
\label{eq:1.12}
\alpha_1 u^3-\alpha_2 u^2-\alpha_3 u-\alpha_4=0,
\end{eqnarray}
where\\
$\alpha_1=w_2(w_4+c)$,\\ $\alpha_2= [w_4(a_2w_1-w_1w_2+2w_2)+c(w_2(2+ D_3)+w_1(w_3+a_2-w_2))]$,\\$\alpha_3=[-w_2(w_4+c)-cD_3w_1(w_3+(a_2-2w_2))-(w_1(a_2-w_2)w_4+c(2D_3w_2+w_1(w_3+a_2-w_2)))]$,\\
$\alpha_4=-cD_3(w_1-1)[w_2+w_1(w_3+a_2-w_2)]$.\\
$u^*$ is the real positive root of \eqref{eq:1.12}. Hence interior equilibrium point $ E_8(u^*,v^*,r^*)$ exists if $({1} -{w_1})<u<{1}, w_1<1$. Knowing the value of $u^*$  we can find the values of $v^*$ and $r^*$ from \eqref{eq:6} and \eqref{eq:1.11} respectively.\\

\begin{remark}
From a realistic biological point of view, we are only interested in the dynamical behavior of model system \eqref{eq:x1}-\eqref{eq:x3} around the positive interior equilibrium point $E_8(u^*,v^*,r^*)$. This ensures the coexistence of all three species. 
\end{remark}

\begin{remark}
Equilibrium points $E_{0},E_{1}, E_{2}, E_{3}$ exist, even though an indeterminate form appears in \eqref{eq:x2od}, \eqref{eq:x3od} due to the ratio dependent functional response. This has been rigorosly proven in \cite{K98}.
\end{remark}

\subsection{Turing conditions}
In this subsection we derive sufficient conditions for Turing instability to occur in \eqref{eq:x1}-\eqref{eq:x3}  where the positive interior equilibrium point $E_8(u^*,v^*,r^*)$ is stable in the absence of diffusion and unstable due to the addition of diffusion, under a small perturbation to $E_8(u^*,v^*,r^*)$. 
This is achieved by first linearizing model \eqref{eq:x1}-\eqref{eq:x3} about the homogenous steady state, by introducing both space and time-dependent fluctuations around $E_8(u^*,v^*,r^*)$. This is given as 
\begin{subequations}\label{eq:7}
\begin{align}
u=u^* +  \hat{u}(\xi,t),\\
v=v^* + \hat{v}(\xi,t),\\
r=r^*  +  \hat{r}(\xi,t),
\end{align}
\end{subequations}
where $| \hat{u}(\xi,t)|\ll u^*$, $| \hat{v}(\xi,t)|\ll v^*$ and  $| \hat{r}(\xi,t)|\ll r^*$. Conventionally we choose 
\[
\left[ {\begin{array}{cc}
\hat{u}(\xi,t)  \\
\hat{v}(\xi,t) \\
\hat{r}(\xi,t)
\end{array} } \right]
=
\left[ {\begin{array}{cc}
\epsilon_1  \\
\epsilon_2 \\
\epsilon_3
\end{array} } \right]
e^{\lambda t + ik\xi},
\]
where  $\epsilon_i$ for $i=1,2,3$ are the corresponding amplitudes, $k$ is the wavenumber, $\lambda$ is the growth rate of perturbation in time $t$ and $\xi$ is the spatial coordinate.
Substituing \eqref{eq:7} into  \eqref{eq:x1}-\eqref{eq:x3} and ignoring higher order terms including nonlinear terms, we obtain the characteristic equation  which  is given as  
\begin{align}\label{eq:1.2.10}
({\bf J} - \lambda{\bf I} - k^2{\bf D})
\left[ {\begin{array}{cc}
\epsilon_1  \\
\epsilon_2 \\
\epsilon_3
\end{array} } \right]=0,
\end{align}
where 
\[
\quad
\bf {D} = 
\left[ {\begin{array}{ccc}
d_1 & 0     & 0  \\
0     & d_2 & 0 \\
0     & 0     & d_3
\end{array} } \right],
\]
$ \bf{J}= \begin{bmatrix}
     u^*\left(-1+\frac{w_1v^*}{(u^*+v^*)^2}\right) &   -\frac{w_1u^{*^2}}{(u^*+v^*)^2} & 0 \\
     \frac{w_2v^{*^2}} {(u^*+v^*)^2}    & v^*\left(-\frac{w_2u^*}{(u^*+v^*)^2}+\frac{w_3r^*}{(v^*+r^*)^2}\right)  & -\frac{w_3v^{*^2}}{(v^*+r^*)^2} \\
        0   & \frac{r^{*^2}(r^*-m)w_4}{(v^*+D_3)^2}  & -\frac{r^*(r^*-m)w_4}{(v^*+D_3)}\\
     \end{bmatrix}
      =\begin{bmatrix}
       J_{11} & J_{12} & J_{13}\\
       J_{21} & J_{22} & J_{23}\\
        J_{31} & J_{32} & J_{33}\\
       \end{bmatrix},
$    \\ 
\\and $\bf{I}$ is a $3\times 3$ identity matrix.\\
For the non-trivial solution of \eqref{eq:1.2.10}, we require that 
\[
\left| 
\begin{array}{ccc}
J_{11}-\lambda -k^2d_1 & J_{12}                            & J_{13}\\
       J_{21}                     & J_{22}-\lambda -k^2d_2 & J_{23}\\
       J_{31}                     & J_{32}                             & J_{33}-\lambda -k^2d_3\\
 \end{array} \right|=0,
\]
which gives a dispersion relation corresponding to $E_8(u^*,v^*,r^*)$. To determine the stability domain associated with $E_8(u^*,v^*,r^*)$, we rewrite the dispersion relation as a cubic polynomial function given as 
\begin{align}\label{eq:1.1.11}
P(\lambda(k^2))=\lambda^3 + \boldsymbol {\mu_2}(k^2)\lambda^2 + \boldsymbol {\mu_1}(k^2)\lambda + \boldsymbol {\mu_0}(k^2),
\end{align}
with coefficients
\begin{align*}
\boldsymbol {\mu_2}(k^2)&=(d_1 + d_2 + d_3)k^2 - (J_{11} + J_{22} + J_{33} ) ,\\
\boldsymbol {\mu_1}(k^2)&=J_{11}J_{33} + J_{11}J_{22} + J_{22}J_{33} - J_{32}J_{23} - J_{12}J_{21}- k^2\big( (d_3 + d_1)J_{22} + (d_2 + d_1)J_{33}  + (d_2 + d_3)J_{11}\big) \\
 &+k^4(d_2d_3 + d_2d_1 + d_1d_3),\\
\boldsymbol {\mu_0}(k^2)&=J_{11} J_{32} J_{23}  - J_{11}J_{22}J_{33} +  J_{12} J_{21} J_{33} + k^2\big(d_1( J_{22} J_{33} - J_{32} J_{23} ) + d_2 J_{11} J_{33} + d_3 ( J_{22} J_{11} - J_{12} J_{21} )\big)\\
& - k^4\big( d_2d_1J_{33} + d_1d_3J_{22} + d_2d_3J_{11}\big) + k^6d_1d_2d_3.
\end{align*}
According to Routh-Hurwitz criterion for stability, $ \mathbb{R}e(\lambda(k^2))<0$ in model \eqref{eq:x1}-\eqref{eq:x3} around equilibrium point $E_8(u^*,v^*,r^*)$ (i.e $E_8$ is stable) if and only if these conditions holds:
\begin{align}\label{eq:1.1.12}
  \boldsymbol {\mu_2}(k^2)>0,\,\boldsymbol {\mu_1}(k^2)>0,\,\boldsymbol {\mu_0}(k^2)>0\quad \text{and}\quad [\boldsymbol {\mu_2}\boldsymbol {\mu_1}-\boldsymbol {\mu_0}](k^2) >0.
\end{align}\\
As violating either of the above conditions implies instability (i.e $\mathbb{R}e(\lambda(k^2))>0$).\\
We now require conditions where an homogenous steady state ($E_8(u^*,v^*,r^*)$) will be stable to small perturbation in the absence of diffusion and unstable in the present of diffusion with certain $k$ vaules. Meaning, we require that around the homogenous steady state $E_8(u^*,v^*,r^*)$ 
\begin{align*}
 \mathbb{R}e(\lambda(k^2>0))>0,\, \text{for some}\, k \, \text{and}\, \mathbb{R}e(\lambda(k^2=0))<0,
\end{align*}
where we consider $k$ to be real and positive even though $k$ can be complex. This behavior is called \emph{diffusion driven instability}. Models that exhibits this sort of behavior in $2$ and $3$ species has been extensively studied in \cite{Gilligan1998}, where several different patterns was observed.\\
In order for homogenous steady state $E_8(u^*,v^*,r^*)$  to be stable( in the absence of diffusion) we need 
\begin{align*}
\boldsymbol {\mu_2}(k^2=0)>0,\,\boldsymbol {\mu_1}(k^2=0)>0,\,\boldsymbol {\mu_0}(k^2=0)>0\quad \text{and}\quad [\boldsymbol {\mu_2}\boldsymbol {\mu_1}-\boldsymbol {\mu_0}](k^2=0) >0,
\end{align*}
where as with diffusion ($k^2>0$) we look for conditions where we can drive the homogenous steady state to be unstable, this can be achieved by studying the coefficient of \eqref{eq:1.1.11}. In order to achieve this we need to reverse at least one of the signs in \eqref{eq:1.1.12}. By which we first study $\boldsymbol {\mu_2}(k^2)$. Irrespective of the value of $k^2$, $\boldsymbol {\mu_2}(k^2)$ will be positive since $J_{11}+J_{22}+J_{33}$ is always less than zero. Therefore we cannot depend on $\boldsymbol {\mu_2}(k^2)$ for diffusion driven instability to occur. Hence for diffusion driven instability to occur in our case, we only depend on the  $2$ conditions which are 
\begin{align}\label{eq:1.1.12a}
\boldsymbol {\mu_0}(k^2)\quad \text{and}\quad [\boldsymbol {\mu_2}\boldsymbol {\mu_1}-\boldsymbol {\mu_0}](k^2).
\end{align}
Both functions are cubic functions of $k^2$, which are generally of the form 
\begin{align*}
G(k^2)=H_H + k^2D_D + (k^2)^2C_C + (k^2)^3B_B,\, \text{with}\, B_B>0,\,\text{and} \, H_H>0.
\end{align*}
The coefficient of $G(k^2)$ are given in Table \ref{tab:coefficientsTable}.
\begin{table}[htb]
\caption{Coefficients of cubic functions $\boldsymbol {\mu_0}(k^2)$ and $[\boldsymbol {\mu_2}\boldsymbol {\mu_1}-\boldsymbol {\mu_0}](k^2)$ used in determining conditions for Turing instability  } 
  \addtolength{\tabcolsep}{-3pt}
  \centering
  {\scriptsize
  \begin{tabular}{@{}llccc@{}}
    \toprule
  Coefficient of $G(k^2)$ & $\boldsymbol {\mu_0}(k^2)$  & $[\boldsymbol {\mu_2}\boldsymbol {\mu_1}-\boldsymbol {\mu_0}](k^2)$ \\ [0.5ex] 
  \hline 
    \multirow{4}{*}{\parbox{1cm}{$H_H$}} & $J_{11}J_{32}J_{23}+ J_{12}J_{21}J_{33}$ & $J_{11}J_{22}J_{33} - (J_{11} + J_{22} + J_{33})(J_{11}J_{22} - J_{12}J_{21} + J_{11}J_{33} + J_{22}J_{33} - J_{23}J_{32})$    \\ 
    &$-J_{11}J_{22}J_{33}$&$ - J_{11}J_{23}J_{32} -J_{12}J_{21}J_{33}$    \\

      \multirow{4}{*}{\parbox{1cm}{$D_D$}}&  & \\&  $d_1( J_{22}J_{33} - J_{32}J_{23} ) + d_2J_{11}J_{33} $ &$d_1( 2J_{11}J_{33} + 2J_{11}J_{22} + 2J_{22}J_{33} + J_{33}J_{33} + J_{22}J_{22} - J_{12}J_{21})$    \\ 
    & $+ d_3( J_{11}J_{22} - J_{12}J_{21} )$ &$+ d_2( 2J_{22}J_{11} + 2J_{22}J_{33} + 2J_{33}J_{11} + J_{11}J_{11} + J_{33}J_{33} - J_{21}J_{12} - J_{23}J_{32})$  \\
        & & $d_3( 2J_{22}J_{11} + 2J_{22}J_{33} + 2J_{33}J_{11} + J_{11}J_{11} + J_{22}J_{22} - J_{23}J_{32})$&\\
        
     \multirow{4}{*}{\parbox{1cm}{$C_C$}} &  &\\&$- d_1d_2J_{33} - d_1d_3J_{22} - d_2d_3J_{11}$ &$-  J_{11}(d_2 + d_3)(2d_1 + d_2 + d_3)-  J_{22}(d_1 + d_3)( d_1 + 2d_2 + d_3)$&    \\ 
    & & $ -  J_{33}(d_1 + d_2)( d_1 + d_2 + 2d_3)$  \\
        
     \multirow{4}{*}{\parbox{1cm}{$B_B$}} &  &\\&$d_1d_2d_3$ &$(d_2 + d_3)( d_1d_1 + d_2d_3 + d_1d_2$&    \\ 
    & & $+ d_1d_3 )$ \\
    \bottomrule
  \end{tabular}}
\label{tab:coefficientsTable} 
\end{table} 
\\To drive either  $\boldsymbol {\mu_0}(k^2)$ or $[\boldsymbol {\mu_2}\boldsymbol {\mu_1}-\boldsymbol {\mu_0}](k^2)$ to negative for some $k$, we  need to find the minimum $k^2$ referred to as the minimum turing point ($k^2_T$) such that $G(k^2=k^2_T)<0$. This minimum turing point occurs when 
$$ {\partial G / \partial (k^2)}=0,$$
this when solved for $k^2$ yields
\begin{align*}
k^2=k^2_T={-C_C + \sqrt{C_C^2 - 3B_BD_D}   \over 3B_B    }.
\end{align*}
To ensures $k^2$ is real and positive such that $ {\partial^2 G / \partial^2 (k^2)}>0$, we require either 
\begin{align}\label{eq:1.1.13}
D_D<0\quad \text{or}\quad C_C<0,
\end{align}
which ensures that 
\begin{align*}
C_C^2 - 3B_BD_D>0.
\end{align*}
Therefore $G(k^2)<0$, if at $k^2=k^2_T$
\begin{align}\label{eq:1.1.14}
G_{min}(k^2)=2C_C^3 - 9D_DC_CB_B - 2(C_C^2- 3D_DB_B)^{3/2} + 27B_BH_H^2<0.
\end{align}
Hence  \eqref{eq:1.1.13}-\eqref{eq:1.1.14} are necessary and sufficient conditions for $E_8(u^*,v^*,r^*)$ to produce diffusion driven instability,  leading to the emergence of patterns. Also to first establish stability when $k=0$, $H_H$  in each case has to be positive. We state this via the following theorem

\begin{theorem}
\label{thm:tur1}
Consider the three species food chain model described via \eqref{eq:x1}-\eqref{eq:x3}, with equilibrium point $E_8(u^*,v^*,r^*)$. Supposing we have a set of parameters such that the following conditions are satisfied,
\begin{enumerate}
\item $H_H>0$ when $k=0$ ,
\item $D_D<0$ or $C_C<0$( and $C_C^2>B_BD_D$),
\item $G_{min}(k^2)=2C_C^3 - 9D_DC_CB_B - 2(C_C^2- 3D_DB_B)^{3/2} + 27B_BH_H^2<0.$
\end{enumerate}

Then Turing instability will always occur in the model.
\end{theorem}

See \cite{Gilligan1998} for the connection between the signs of the coefficients in table \ref{tab:coefficientsTable}  and the associated patterns.

\subsection{Effect of Allee threshold parameter $m$ on Turing Instability}
In this section we investigate the effects of Allee threshold $m$ on Turing patterns. We focus on whether $m$ can induce Turing instabilities. We perform numerical simulation of model system \eqref{eq:x1}-\eqref{eq:x3} and the use of coefficient of dispersion relation, given in Table \ref{tab:coefficientsTable}.\\
A Chebychev collocation scheme is employed for the approximation of
the time varying, one-dimensional
problem\cite{timeDependentSpectralMethods,doi:10.1137/0720063}.

Our one-dimensional numerical simulation employs the zero-flux boundary condition with a spatial domain of size $(256\times 256)$. 
We define the spatial domain to be $X \in (0,L)$, in the nondimensionalistaion in appendix \ref{app1}, where $\quad x={X\pi\over L}$.

The initial condition as defined is a small perturbation around the positive homogenous steady state given as 
\begin{align*}
u=u^{*} + \epsilon_1  cos^2(nx)(x > 0)(x < \pi),\\
v=v^{*}  + \epsilon_2 cos^2(nx)(x > 0)(x < \pi),\\
r=r^{*}  + \epsilon_3  cos^2(nx)(x > 0)(x < \pi),
\end{align*}
where $ \epsilon_i=0.05$ $\forall i$. 

In the two-dimensional case, we numerically solve model system \eqref{eq:x1}-\eqref{eq:x3}  using a finite difference method. A forward difference
scheme is used for the reaction part. A standard five-point explicit finite difference scheme is used for the two-dimensional diffusion terms. Model system \eqref{eq:x1}-\eqref{eq:x3}  is numerically solved in two-dimension over $200\times 200$ mesh points with spatial resolution $\Delta x=\Delta y=0.1$ and time step $\Delta t=0.1$. \\
In our numerical simulation both in the one-dimensional and two-dimensional case, with specific parameter set as show in figure \eqref{fig:turing1}-figure \eqref{fig:turing5}, different types of dynamics are observed as a result of the value of the Allee threshold $m$. \\
From figure \eqref{fig:turing1}-figure \eqref{fig:turing2}, we numerically observe that change in $m$ can result in disappearance of patterns or the appearance of patterns. Specifically from figure \eqref{fig:turing1}-figure \eqref{fig:turing2}, patterns where observed when $m=0.01$ but these patterns change as $m$ gets larger and finally at $m=0.5$, these patterns completely disappear. The critical Allee threshold $m$, for patterns to disappear, depends on the parameter set being used. In some cases, patterns where observed when $m=0$ but as $m\to 0.2$ and beyond, these patterns disappeared.  \\
\begin{figure}[!htp]
	\begin{center}
	\includegraphics[scale=0.17]{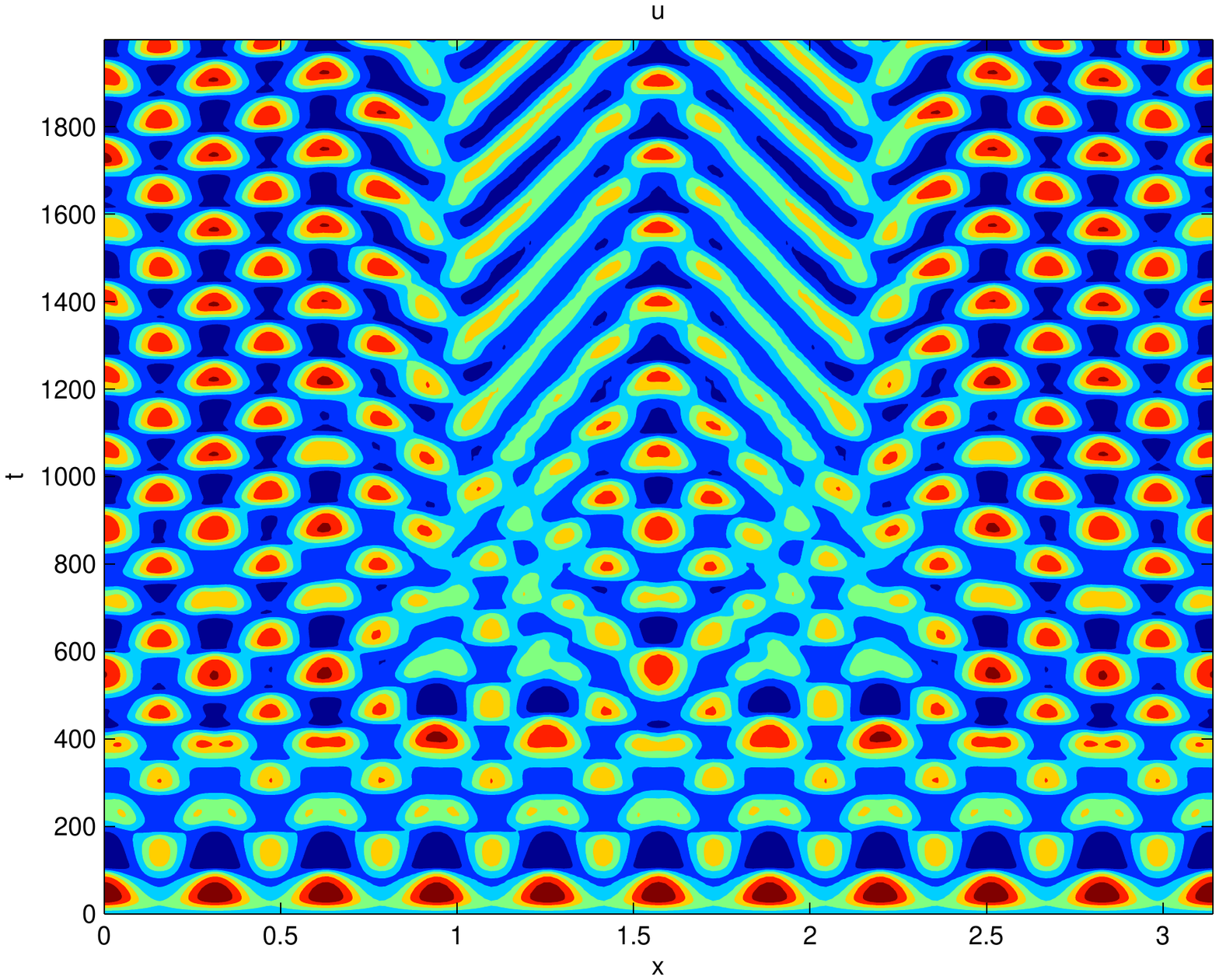}
    \includegraphics[scale=0.17]{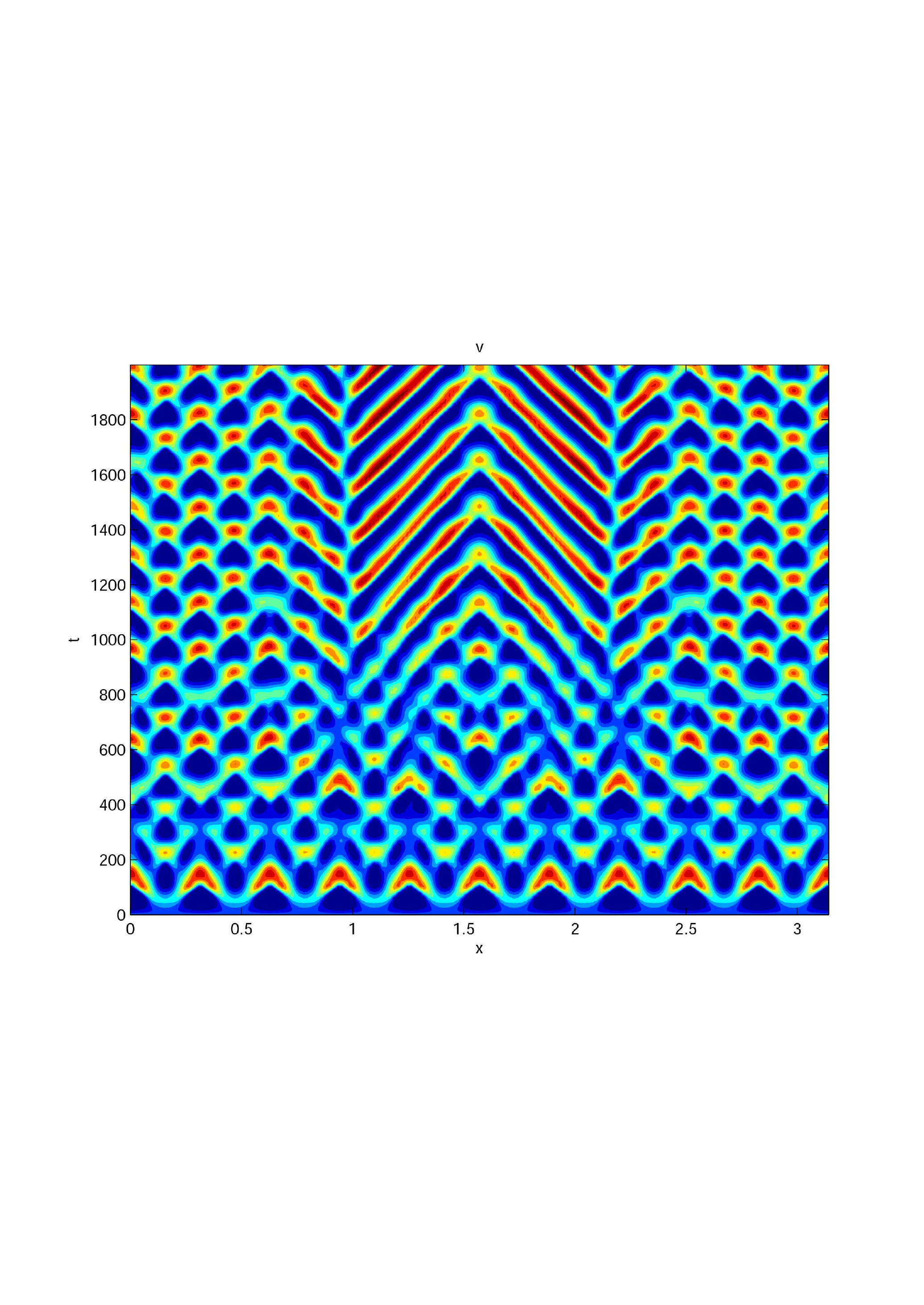}
	\includegraphics[scale=0.17]{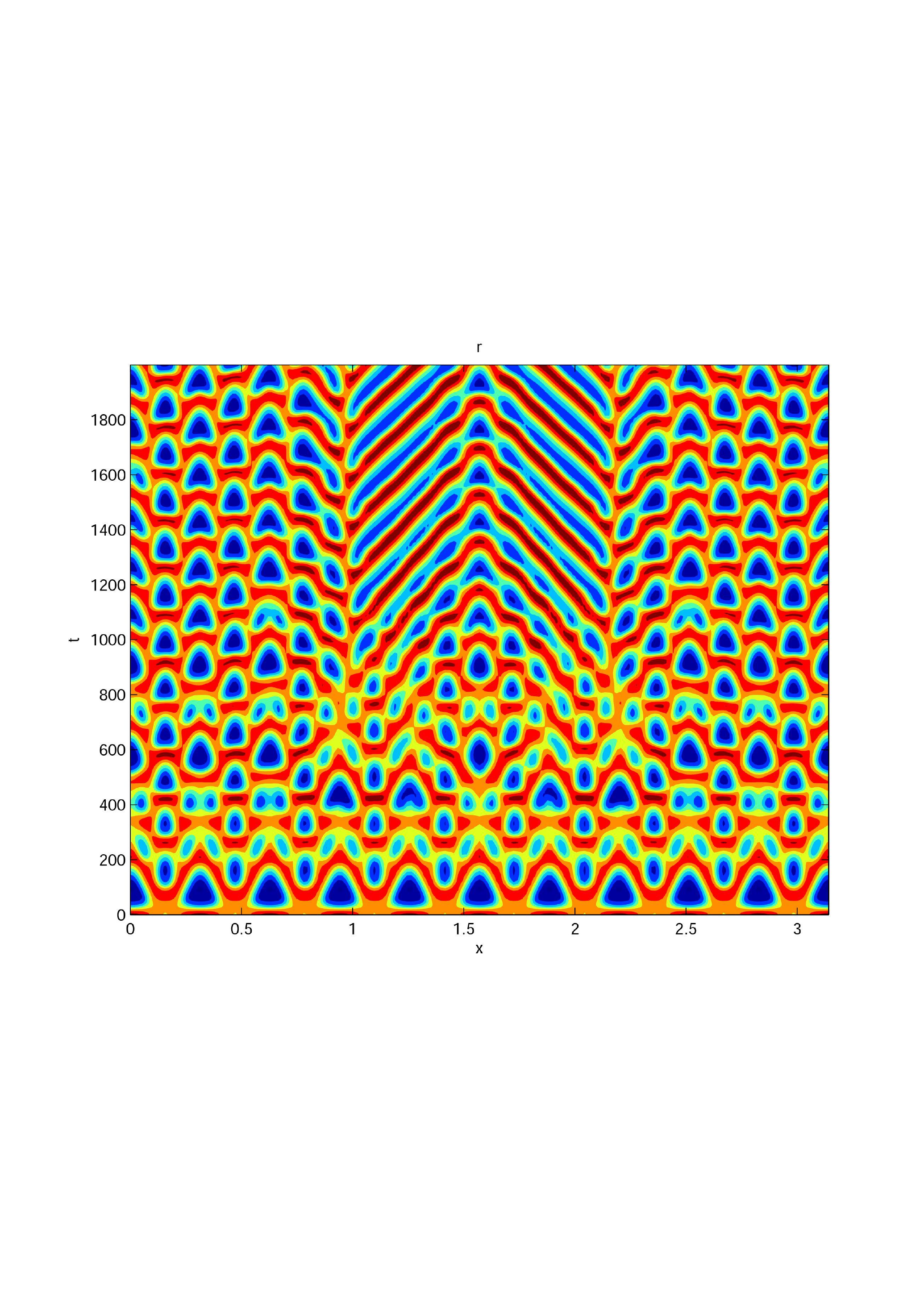}
	\end{center}
		\caption{The densities of the three species are shown as contour plots in the x-t plane (1 dimensional in space). The long-time simulation yields stripes-spots mixture pattern, that are spatio-temporal. The  parameters are: $ m=0.01, w_1 = 0.96, w_2 = 0.52, w_3 = 1.06, w_4 = 0.37, a_2 = 0.014,D_3 = 0.1, c = 0.1, d_1 =10^{-3},  d_2 = 10^{-5}, d_3 =10^{-5}$. }
	\label{fig:turing1}
\end{figure}
\begin{figure}[!htp]
	\begin{center}
	\includegraphics[scale=0.17]{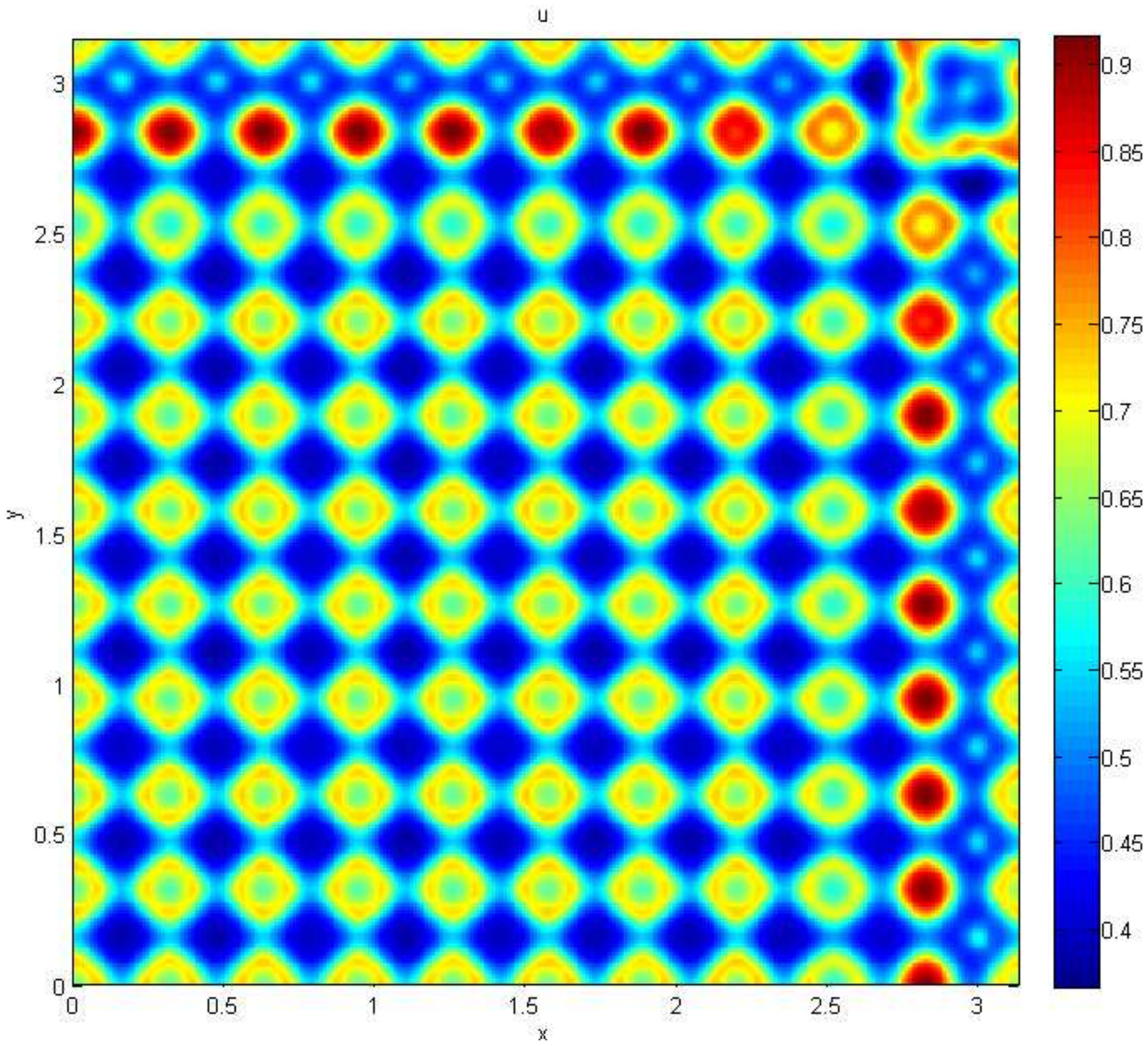}
    \includegraphics[scale=0.17]{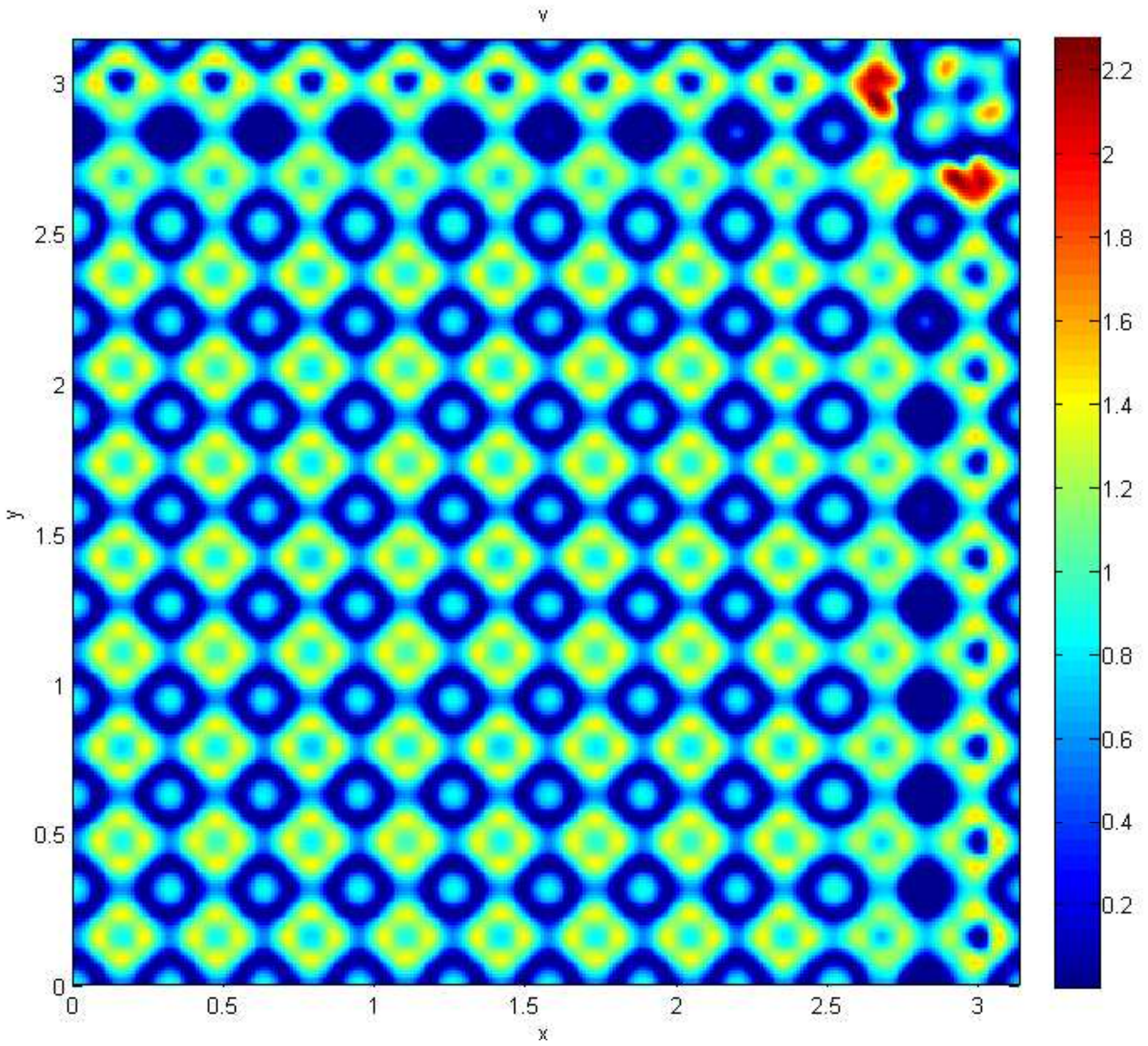}
    \includegraphics[scale=0.17]{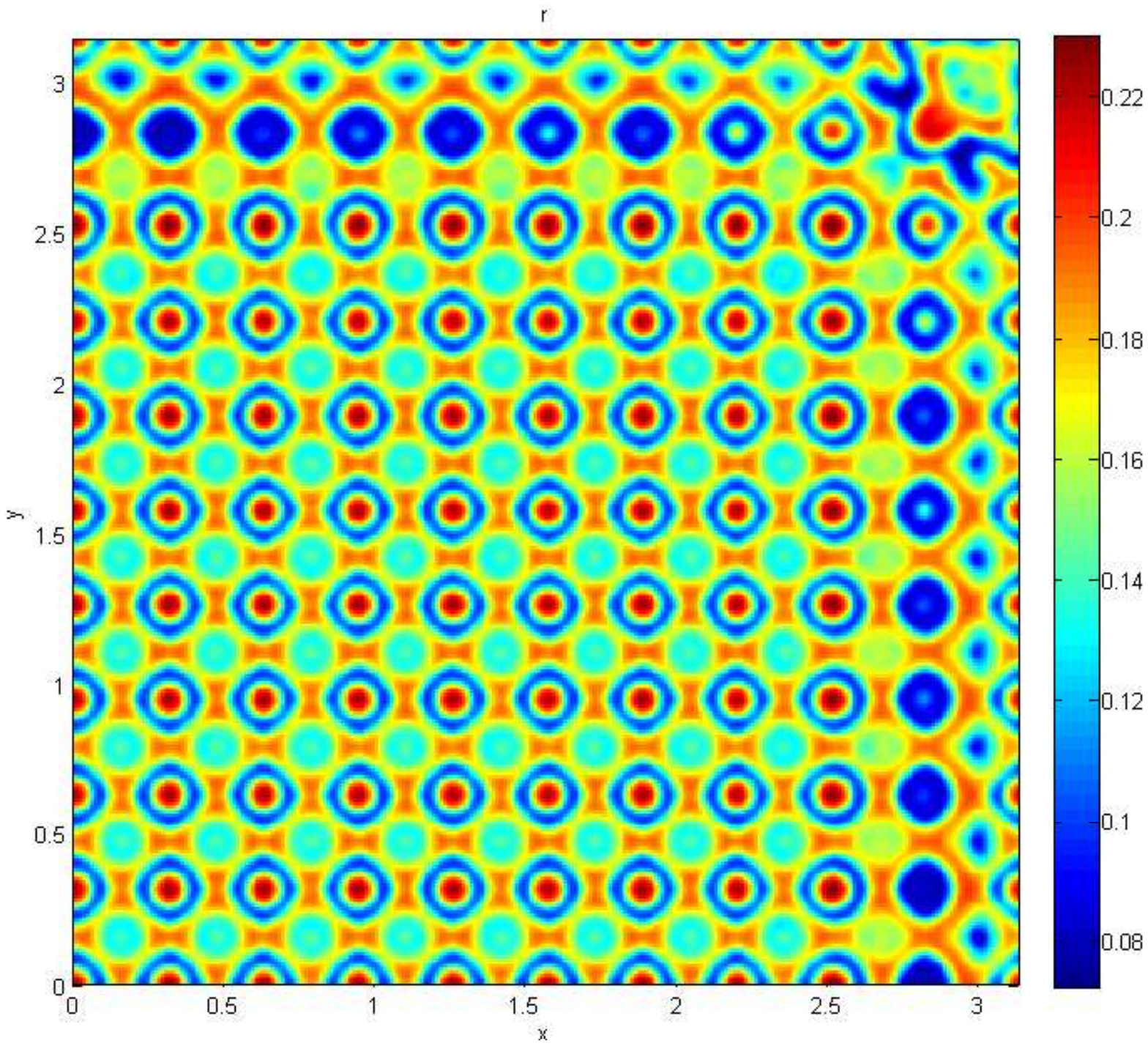}
	\end{center}
		\caption{The densities of the three species are shown as contour plots in the x-y plane (2 dimensional in space). The long-time simulation yields spots pattern. The  parameters are: $ m=0.01, w_1 = 0.96, w_2 = 0.52, w_3 = 1.06, w_4 = 0.37, a_2 = 0.014,D_3 = 0.1, c = 0.1, d_1 =10^{-3},  d_2 = 10^{-5}, d_3 =10^{-5}$. }
	\label{fig:turing2}
\end{figure}
It is also conclusive from figure \eqref{fig:turing3}-figure \eqref{fig:turing5} that Allee threshold has an effect on the type of patterns that do form particularly, as in the case of figure \eqref{fig:turing3} and \eqref{fig:turing5}, changing $m$, can change spatio-temporal patterns to a fixed spatial patterns. This behavior can be observed using the dispersion relation as shown in figure \eqref{fig:Disp1}. 
In choosing $\mu_0(k^2)<0$ for some $k$, which will mostly result in fixed spatial patterns, changing the value of $m$ does not result in a change in the type of patterns formed, but parameters move out of the Turing space. Therefore in our numerical experiments we do not find parameters such that  fixed spatial patterns can change to spatial-temporal patterns, as $m$ increases from zero.\\
\begin{figure}[!htp]
	\begin{center}
	\includegraphics[scale=0.17]{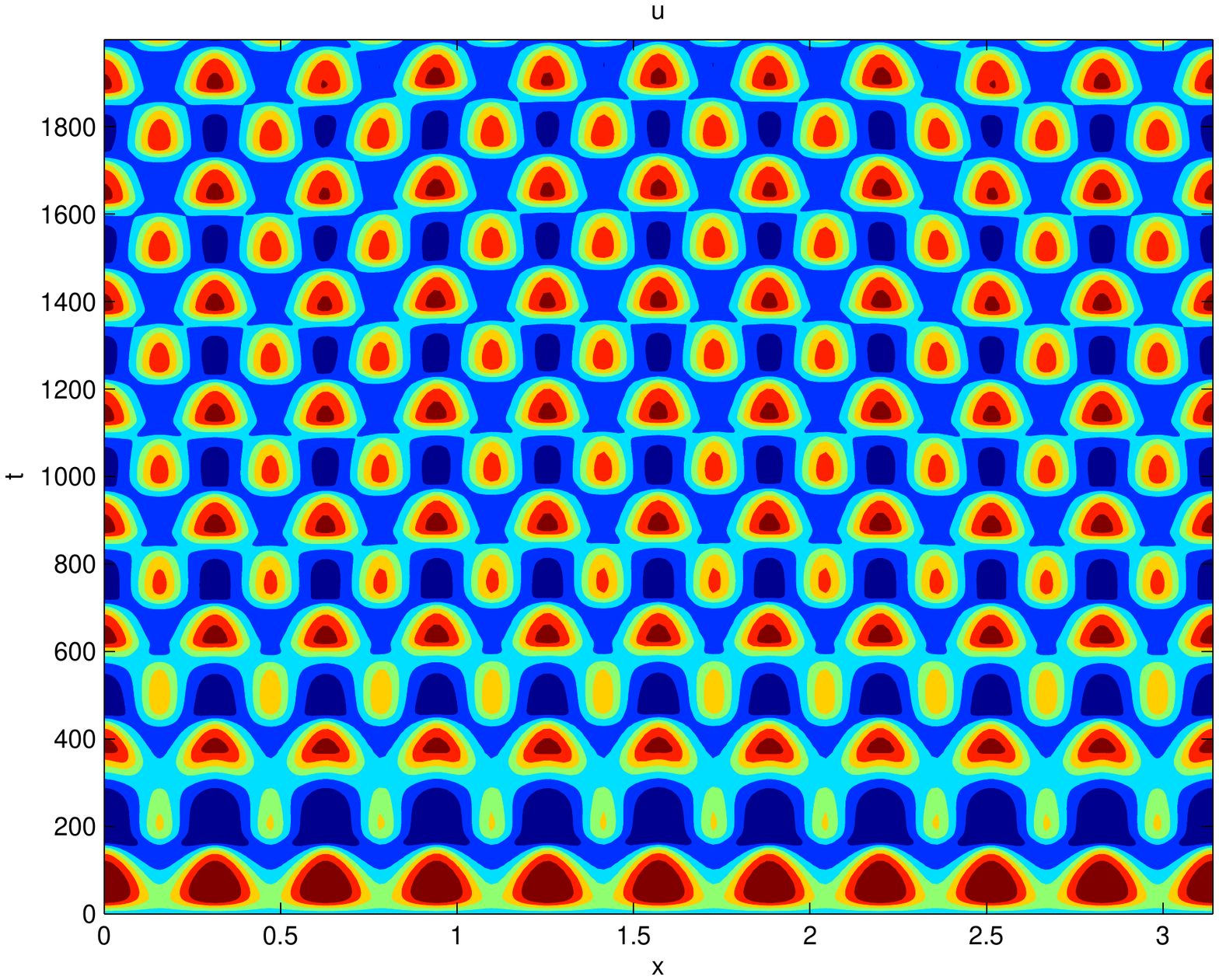}
    \includegraphics[scale=0.17]{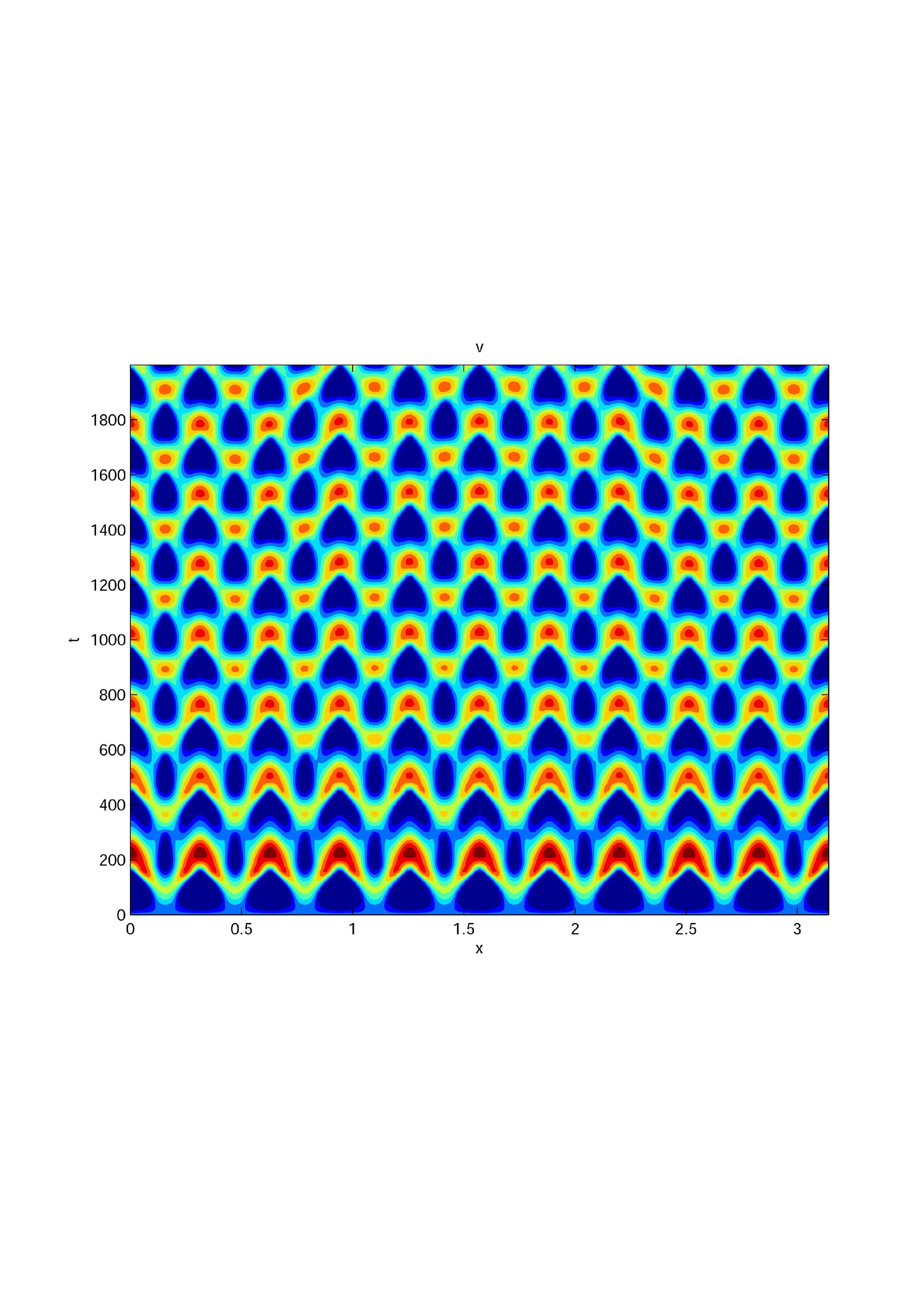}
	\includegraphics[scale=0.17]{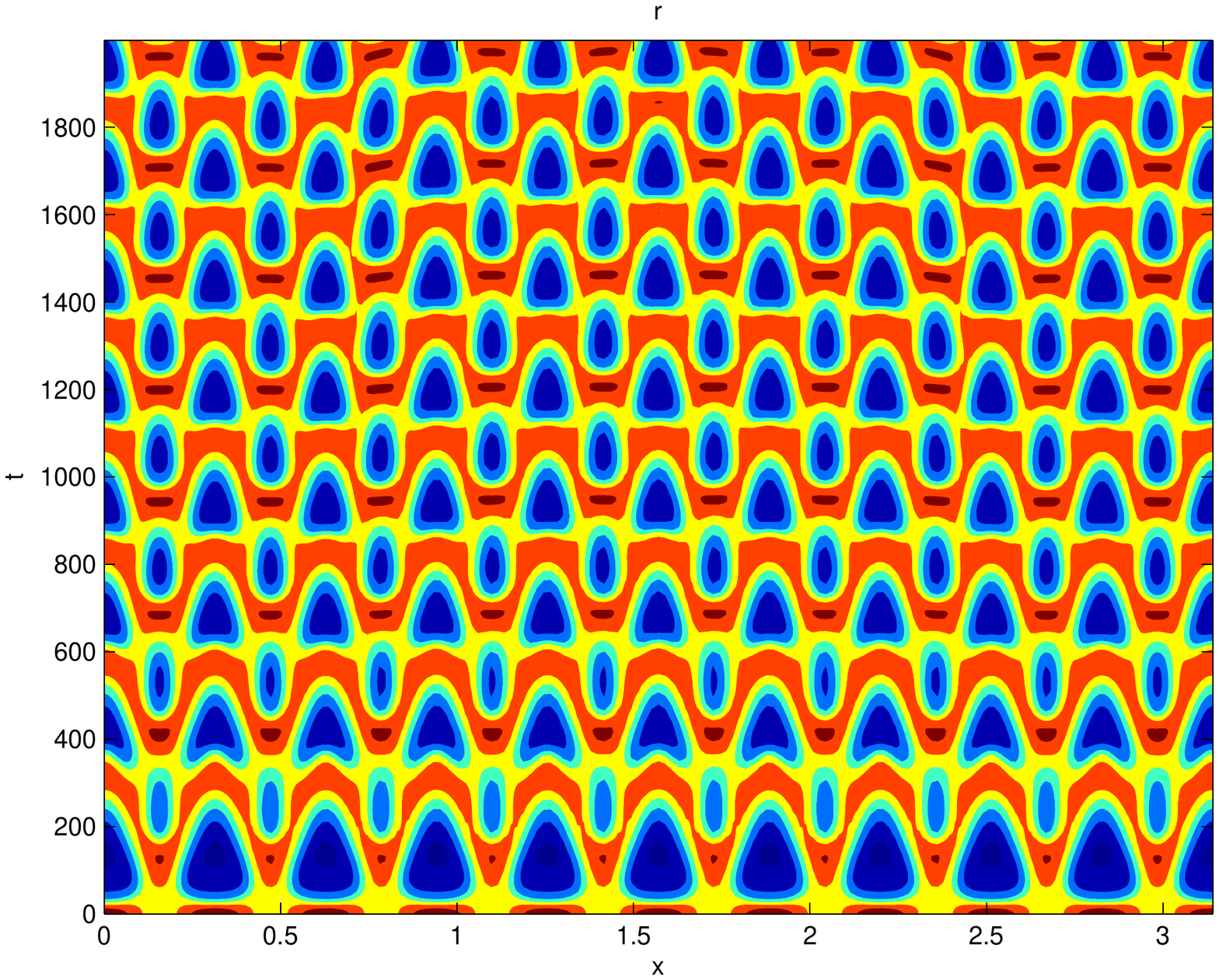}
	\end{center}
		\caption{The densities of the three species are shown as contour plots in the x-t plane (1 dimensional in space). The long-time simulation yields spot and arrow type patterns. The  parameters are: $ m=0, w_1 = 0.95, w_2 = 0.3, w_3 = 0.82, w_4 = 0.53, a_2 = 0.01,D_3 = 0.1, c = 0.1, d_1 =10^{-3},  d_2 = 10^{-5}, d_3 =10^{-5}$}
	\label{fig:turing3}
\end{figure}\\
\begin{figure}[!htp]
	\begin{center}
	\includegraphics[scale=0.17]{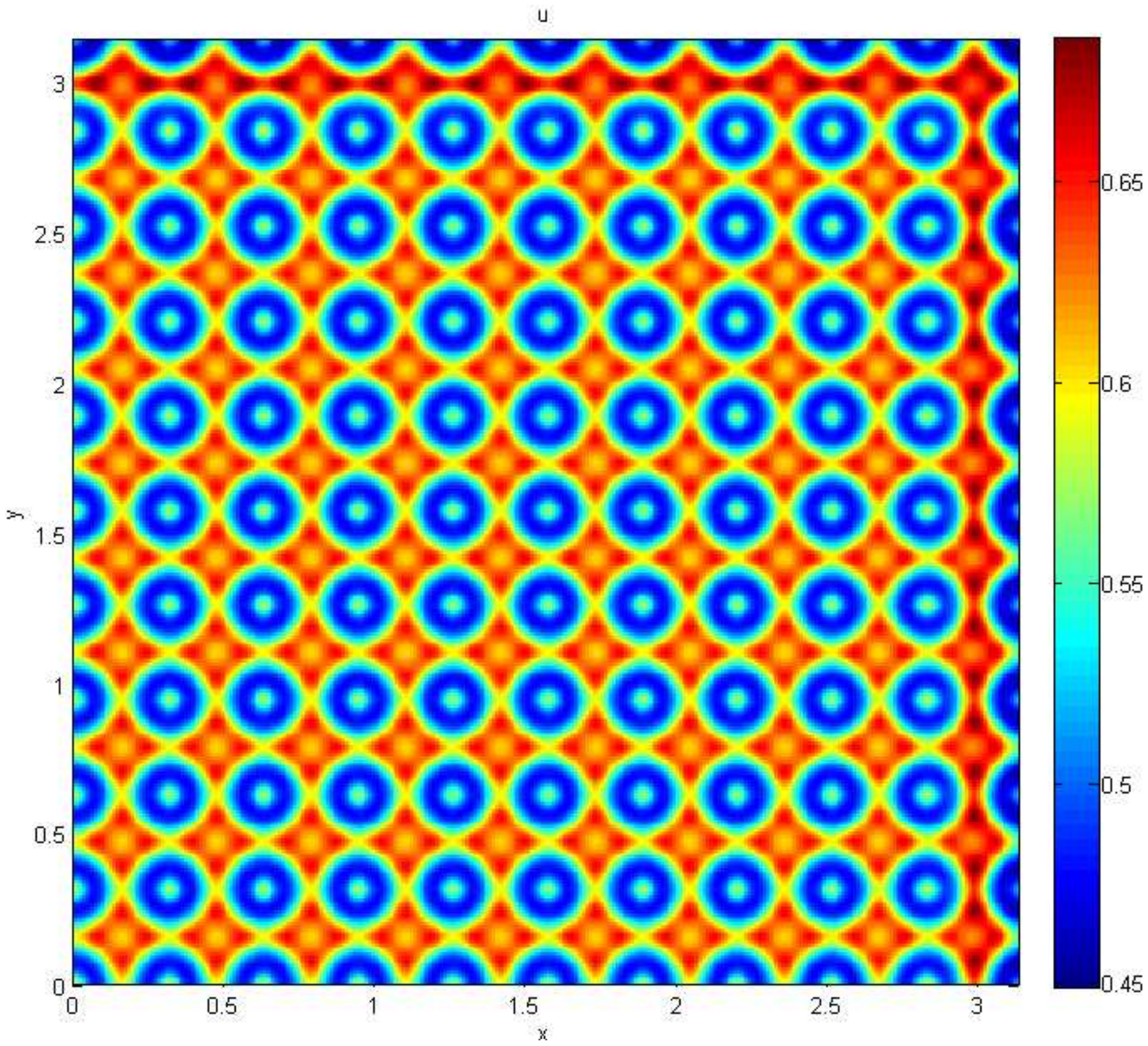}
    \includegraphics[scale=0.17]{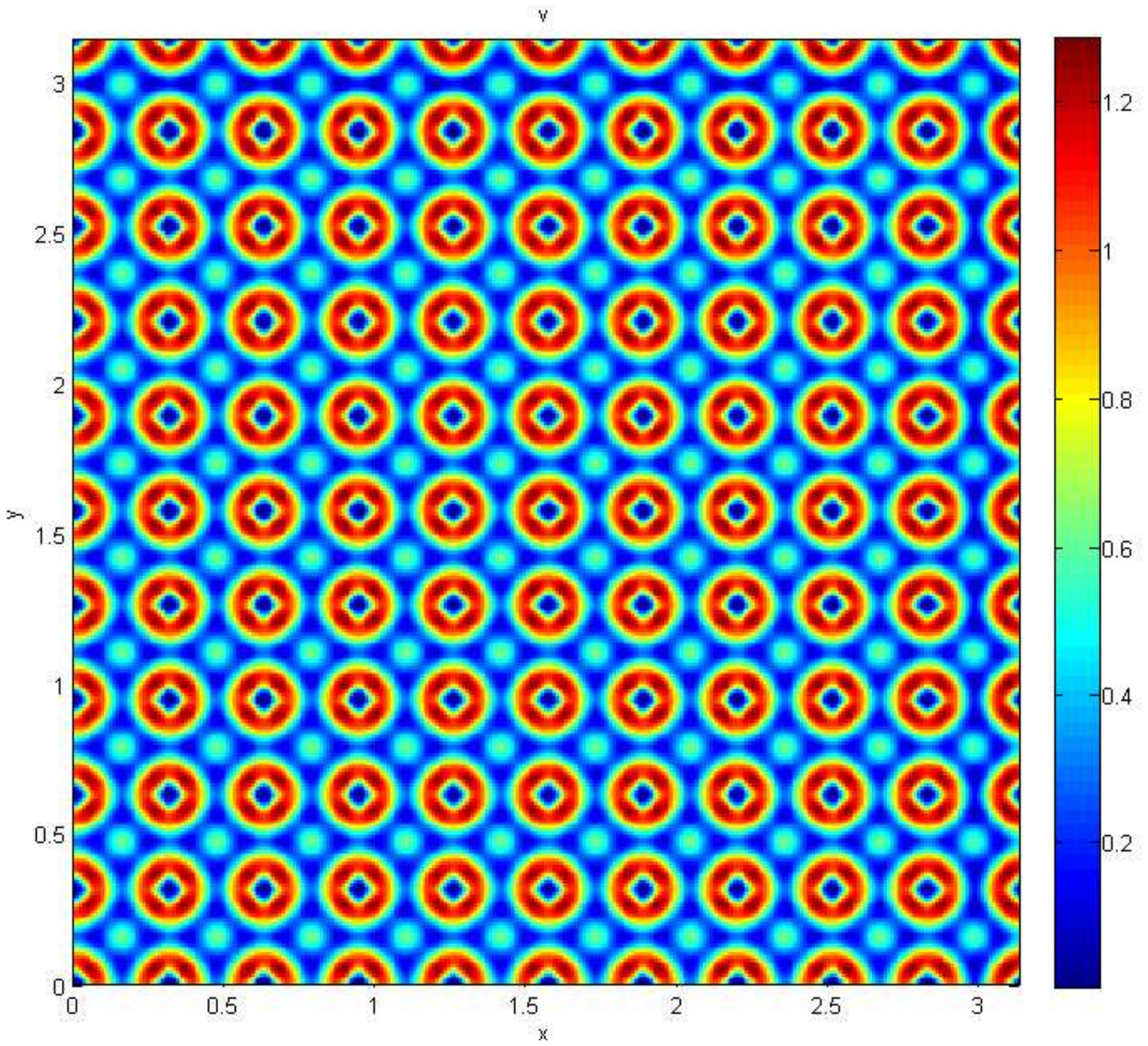}
	\includegraphics[scale=0.17]{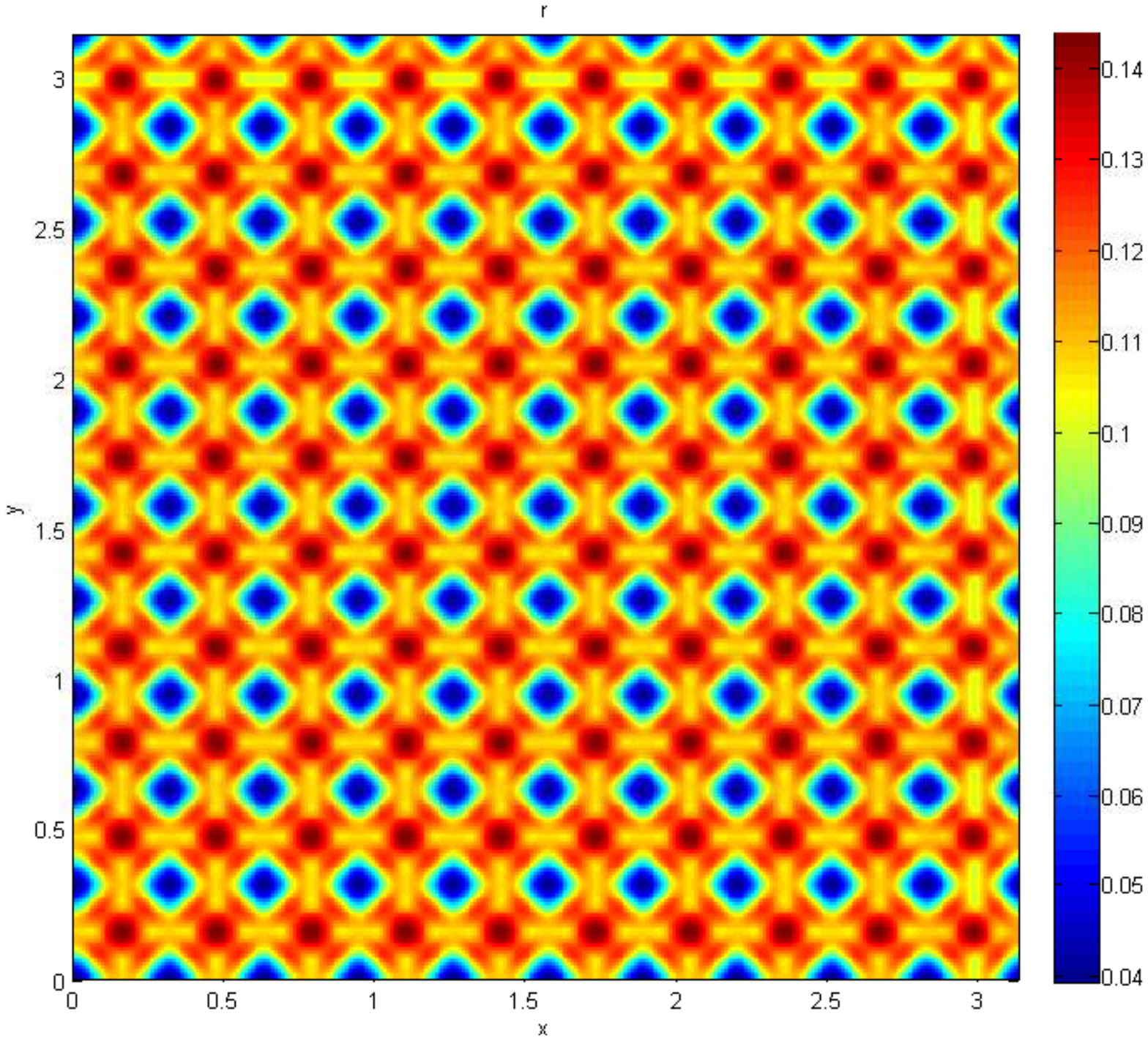}
	\end{center}
		\caption{The densities of the three species are shown as contour plots in the x-y plane (2 dimensional in space). The long-time simulation yields spot patterns. The  parameters are: $ m=0, w_1 = 0.95, w_2 = 0.3, w_3 = 0.82, w_4 = 0.53, a_2 = 0.01,D_3 = 0.1, c = 0.1, d_1 =10^{-3},  d_2 = 10^{-5}, d_3 =10^{-5}$}
	\label{fig:turing4}
\end{figure}\\
\begin{figure}[!htp]
	\begin{center}
	\includegraphics[scale=0.17]{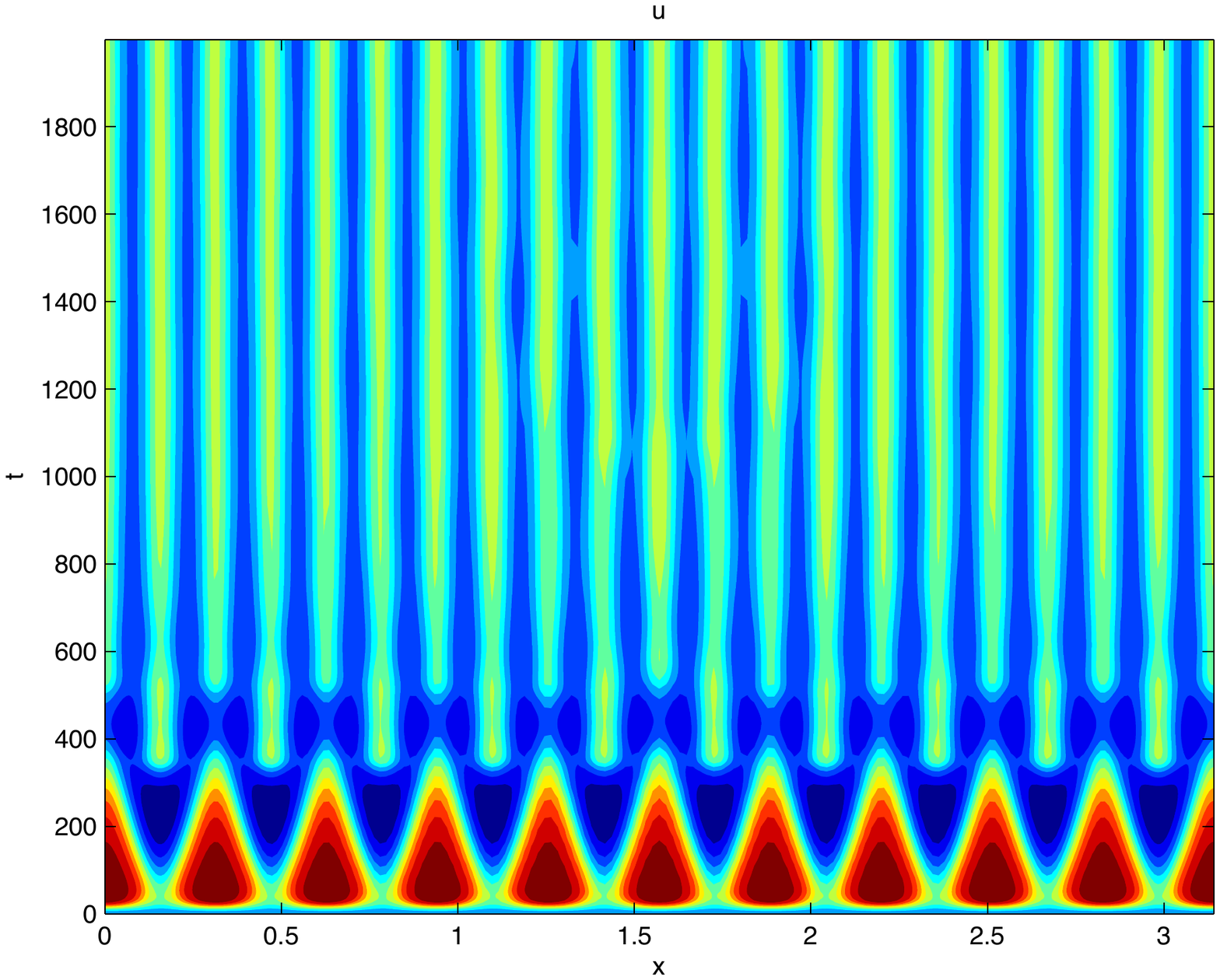}
    \includegraphics[scale=0.17]{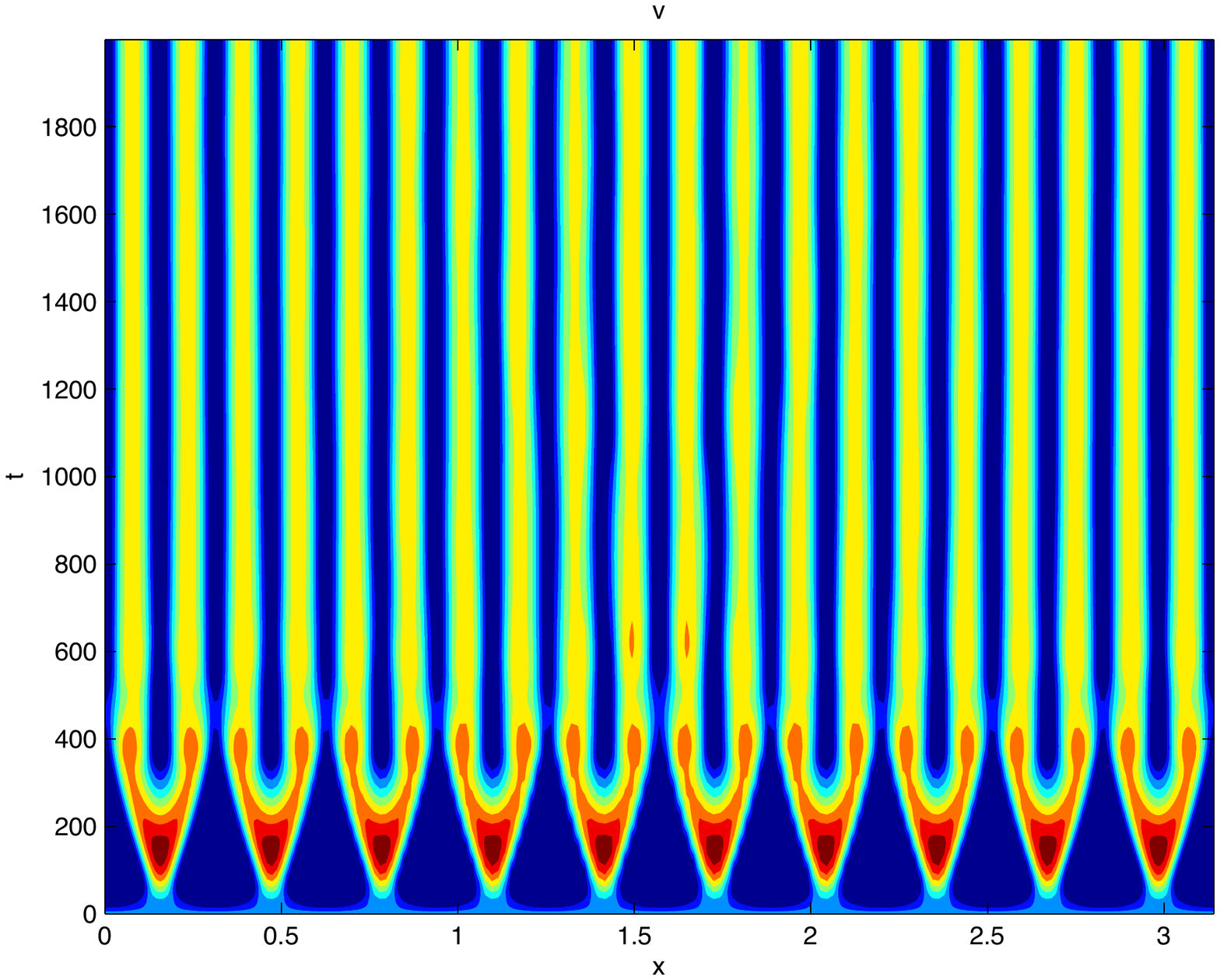}
	\includegraphics[scale=0.17]{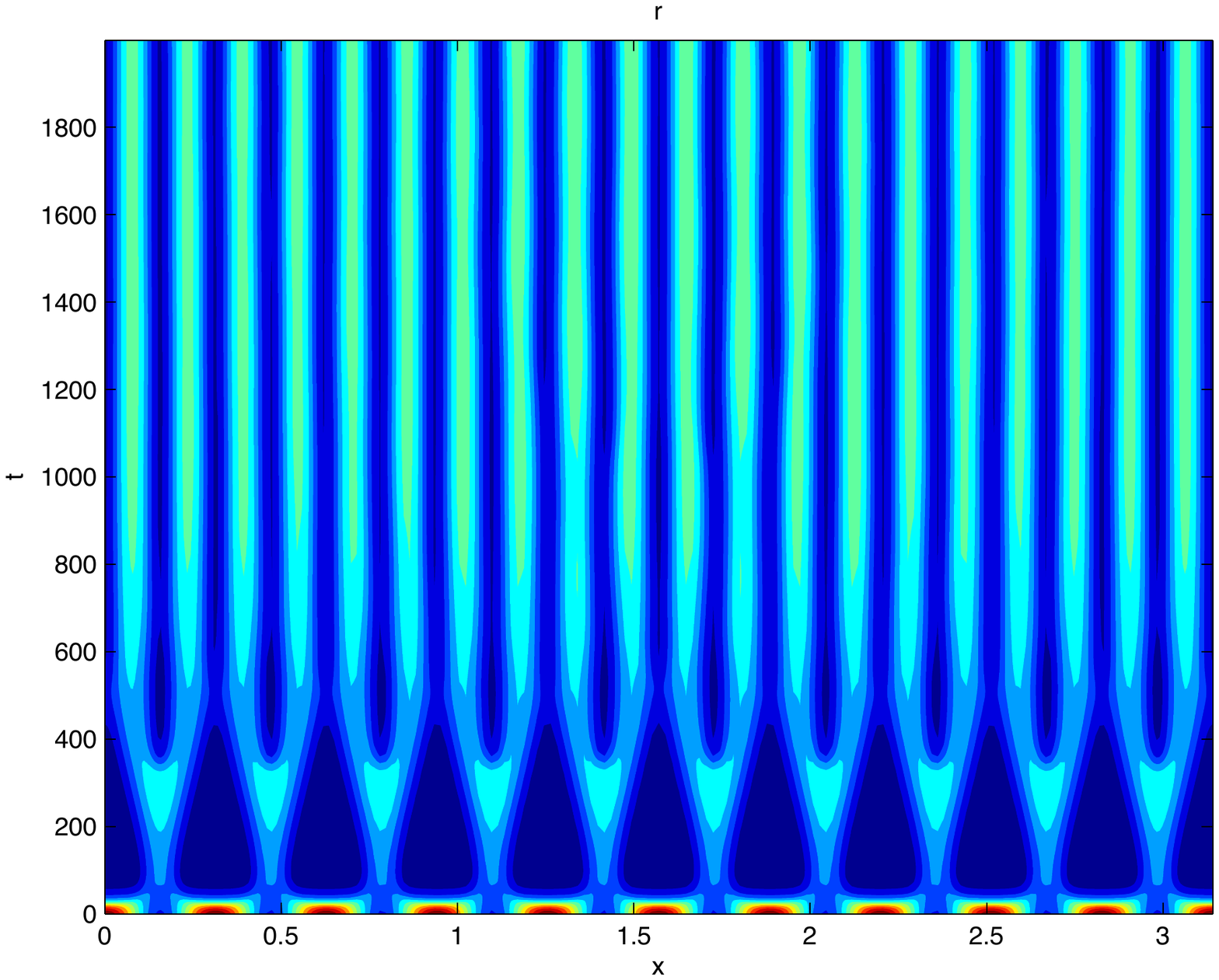}
	\end{center}
		\caption{The densities of the three species are shown as contour plots in the x-t plane (1 dimensional in space). The long-time simulation yields strip patterns. The  parameters are: $ m=0.1, w_1 = 0.95, w_2 = 0.3, w_3 = 0.82, w_4 = 0.53, a_2 = 0.01,D_3 = 0.1, c = 0.1, d_1 =10^{-3},  d_2 = 10^{-5}, d_3 =10^{-5}$ }
	\label{fig:turing5}
\end{figure}
\begin{figure}[!htp]
	\begin{center}
	\includegraphics[scale=0.20]{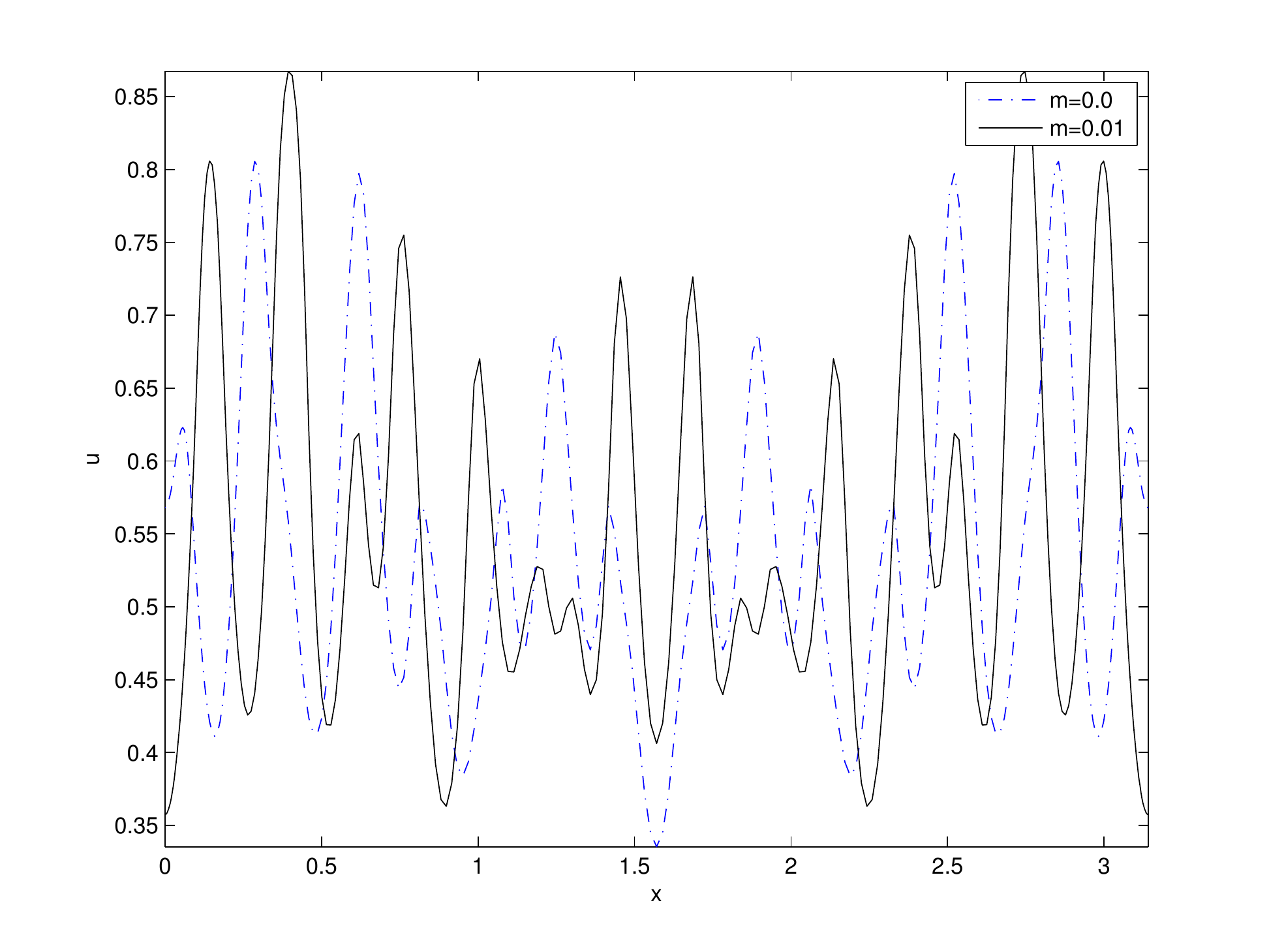}
    \includegraphics[scale=0.20]{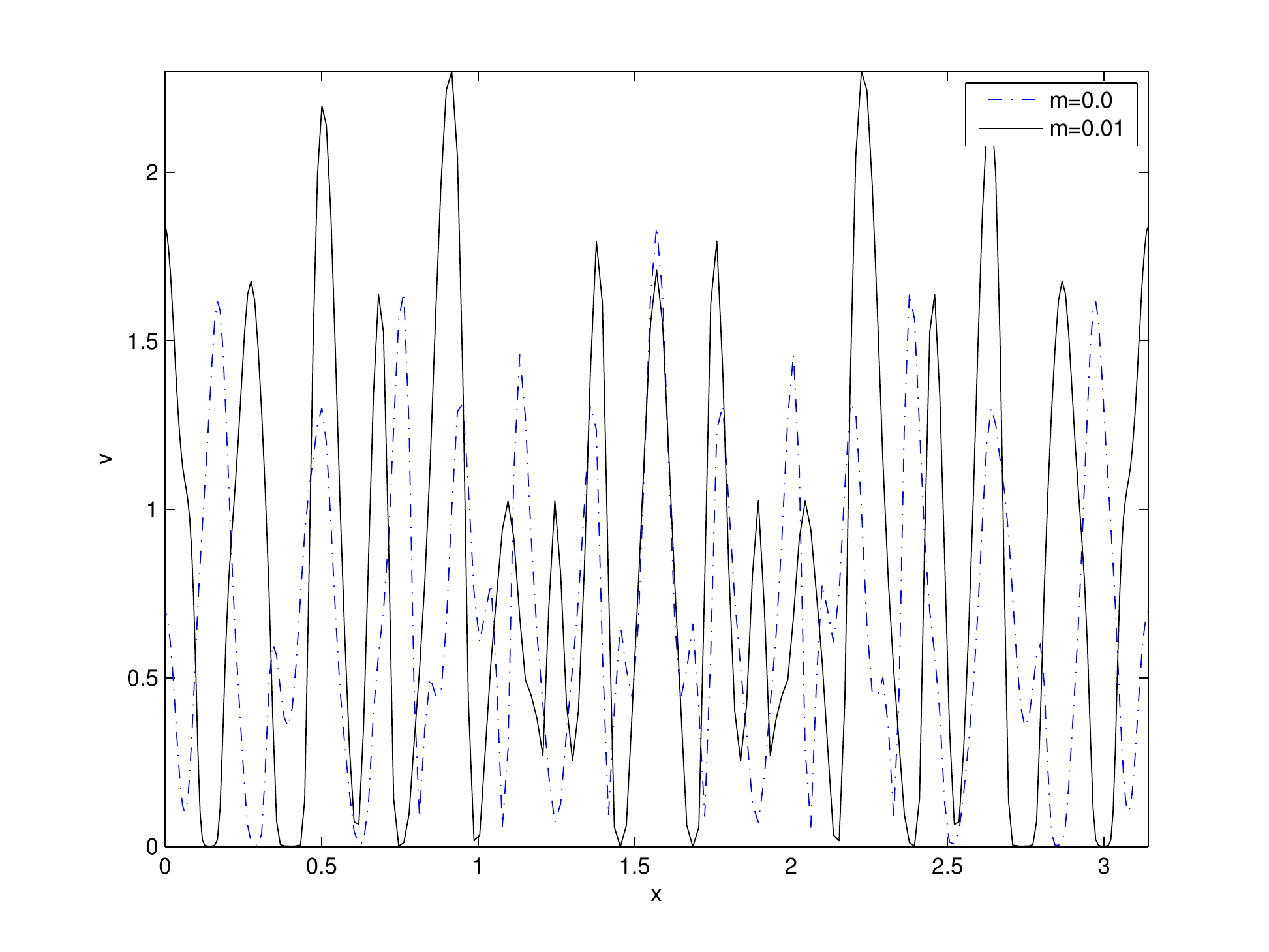}
	\includegraphics[scale=0.20]{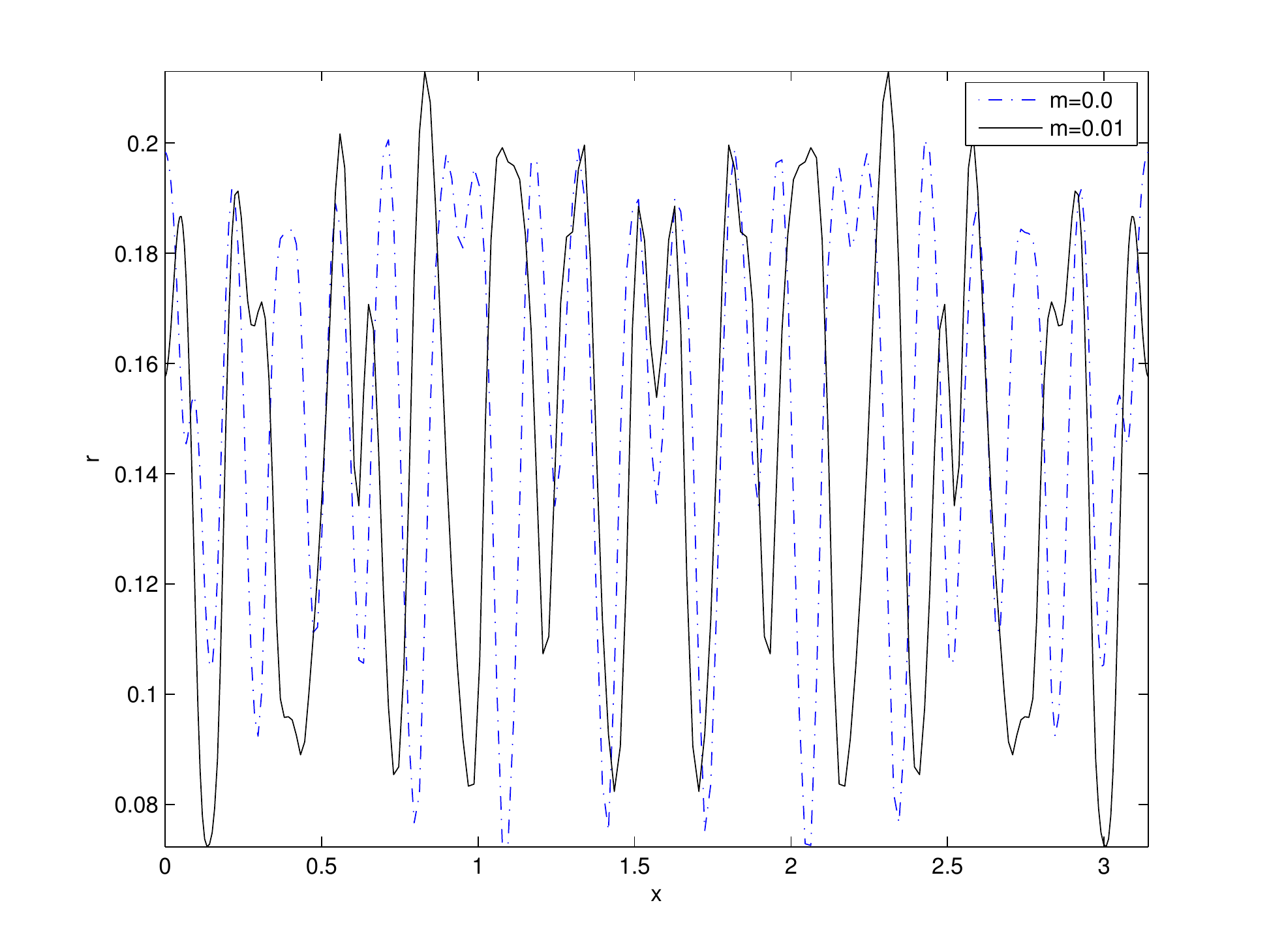}
	\end{center}
		\caption{The density of the three species shown in figure \ref{fig:turing1} has $m=0.01$. We want to compare this density to the case where $m=0$. Thus we run a simulation with $m=0$, and the same parameters in figure \ref{fig:turing1}, $w_1 = 0.95, w_2 = 0.3, w_3 = 0.82, w_4 = 0.53, a_2 = 0.01,D_3 = 0.1, c = 0.1, d_1 =10^{-3},  d_2 = 10^{-5}, d_3 =10^{-5}.$}
We next observe the densities at time t=1500 for $m=0$ (dashed line) vs $m=0.01$ (solid line), to see how the species density changes in space as the Allee threshold $m$ changes. 
	\label{fig:turing45}
\end{figure}\\
\begin{figure}[!htp]
	\begin{center}
	\includegraphics[scale=0.22]{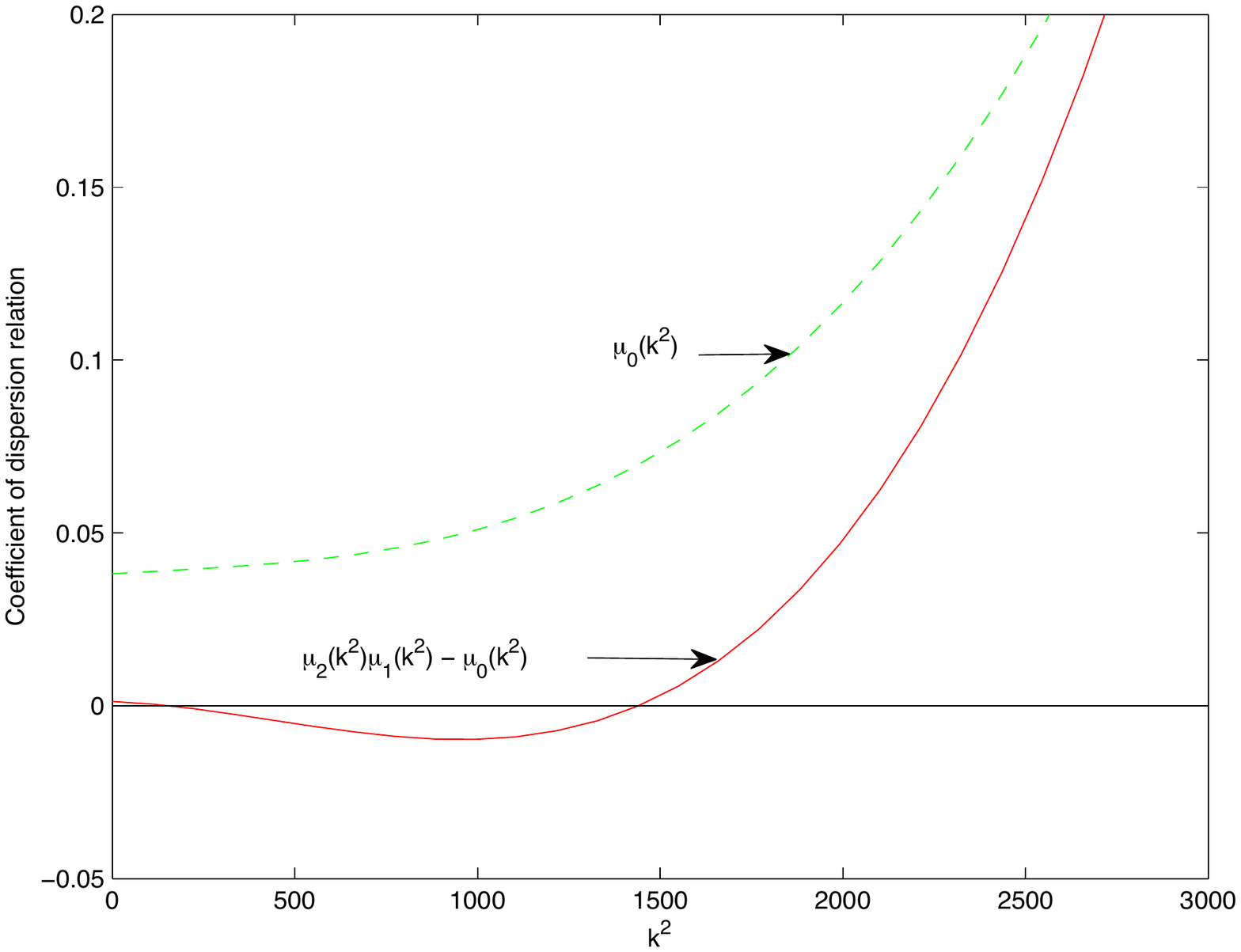}
    \includegraphics[scale=0.22]{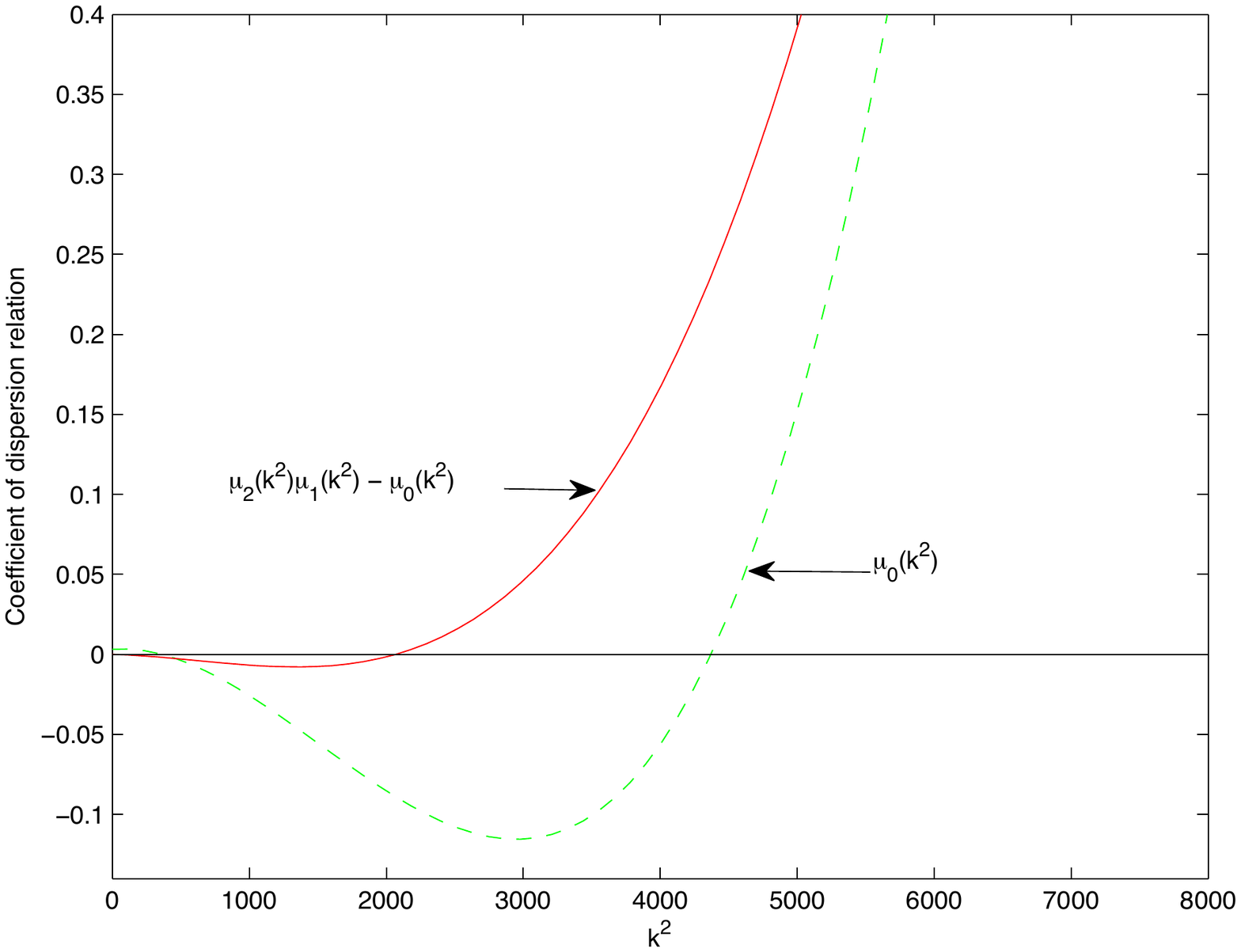}
	\end{center}
		\caption{Plot of coefficient of dispersion relation with $m=0$ and $m=0.1$ for figure \eqref{fig:turing3} and figure \eqref{fig:turing4} respectively. }
	\label{fig:Disp1}
\end{figure}
Table \eqref{tab:ODE Stability}, also displays the fact that no patterns were observed as $m$ increases from $m=0.01$ to $m=0.5$ from  figure  \eqref{fig:turing1} and figure \eqref{fig:turing2} because $m$ can destabilizes an already stable equilibrium point in the ode model as in when $k=0$. Hence as in figure  \eqref{fig:turing1}  and figure \eqref{fig:turing2} where $m=0.01$ produces patterns since $E_8(u^*,v^*,r^*)$ is stable when $k=0$, this stable ($+$) equilibrium point becomes unstable ($-$) even when $k=0$. 
\begin{table}[htb]
\begin{center}
\begin{tabular*}{.935\textwidth}{| c | c | c |c | }
  \hline\hline
    $m$ & $H_H=\mu_0(k^2=0)$  & $H_H=[\mu_2\mu_1-\mu_0](k^2=0)$& Pattern\\ \hline
    \hline
    $0$       & $+$   & $+$ &   Patterns may occur  \\ \hline
    $0.1$    & $+$   & $+$ &   Patterns may occur \\ \hline
    $0.2$    & $-$    & $-$  &   No Patterns \\ \hline
    $0.3$    & $-$    & $-$  &   No Patterns \\ \hline
    $0.4$    & $-$    & $-$  &   No Patterns \\ \hline
    $0.5$    & $-$    & $-$  &   No Patterns  \\ \hline
\end{tabular*}
\end{center}
		\caption{The influence of $m$ on the sign of $\boldsymbol {\mu_0}(k^2)$ and $[\boldsymbol {\mu_2}\boldsymbol {\mu_1}-\boldsymbol {\mu_0}](k^2)$. Describing the stability of $E_8(u^*,v^*,r^*)$ when $k=0$  and how its might lead to pattern formation. The  parameters are: $w_1 = 0.96, w_2 = 0.52, w_3 = 1.06, w_4 = 0.37, a_2 = 0.014,D_3 = 0.1, c = 0.1, d_1 =10^{-3},  d_2 = 10^{-5}, d_3 =10^{-5}$.}
		    \label{tab:ODE Stability}
\end{table}

\section{Discussion}

In this manuscript we have considered a spatially explicit ratio dependent three species food chain model, with a strong Allee effect in the top predator. 
The model represents dynamics of three interacting species in ``well-mixed" conditions and is applicable to ecological problems in terrestrial, as well as aquatic systems. Our goal is to understand the effect of the Allee threshold on the system dynamics, as there is not much work on Allee effects in multi-trophic level food chains, and we hope to address this here.  

We first prove the existence of a global attractor for the model. In many systems an interesting question is to consider the effect on the global attractor, under perturbation of the physical parameters in the system. Most often continuity cannot be proven, but upper semi-continuity can.
We show that the global attractor is upper semi-continuous w.r.t the Allee threshold parameter $m$, that is 
$dist_{E}(\mathcal{A}_{m},\mathcal{A}_{0}) \rightarrow 0, \ \mbox{as} \ m \rightarrow 0^{+}$, via theorem \ref{thm:gaus1sc}. 
To the best of our knowledge this is the first robustness result for a three-species model with strong Allee effect. 
Unless a systems dynamics are robust, there is no possibility to capture the same ecological behavior in a laboratory experiment, or natural setting.

The next question of interest, after one has proved such a result, is estimating the rate of decay, in terms of the physical parameter in question. Since there are no theoretical results to estimate the decay rate to the target attractor (that is when $m=0$), we choose to investigate this rate numerically.
We find decay estimates of the order of $\mathcal{O}(m^{\gamma})$, where $\gamma$ is found explicitly to be slightly less, but very close to 1.

We also find that $m$ effects the spatiotemporal dynamics of the system in two distinct ways.
\newline
(1) It changes the Turing patterns that occur in the system, in both one and two spatial dimensions. 
That is varying $m$, has a distinct effect on the sorts of Turing patterns that form.
This can have interesting effects on the patchiness of spatially dispersing species. Our results say that if one can effect the Allee threshold in the top predator, one can effectively change the spatial concentrations of the species involved. This is best visualized in figure \ref{fig:turing45}. This could have many potential applications in Allee mediated biological control, and conservation, as recently biologists have begun to consider Allee effects as beneficial in limiting establishment of an invading species \cite{tobin2011exploiting}.
\newline
(2) It facilitates overexploitation phenomenon. That is without $m$, or when $m=0$, one does not see extinction in top predator, as the origin becomes unstable. However,
introducing it stabilises the origin, and low initial concentrations of $r$, may yield extinction, via theorems \ref{thm:ox}. This has many possible applications to top down trophic cascades and cascade control. This also tells us that in a three species system, where the top predator can switch its food source, only having an Allee effect in place can possibly drive it to extinction, thus bring about true overexploitation. This could have potential applications in Allee mediated biological control, if the top predator is an invasive species.

As future investigations it would be interesting to try to model weak Allee effects in top predator, and/or Allee effects in the other species as well, be they weak or strong. It has been stipulated with evidence, that two or more Allee effects can occur simultaneously in the same population\cite{berec2007multiple}. Thus an extremely interesting question would be to consider different type of Allee effects in the different species, and investigate the interplay of the various Allee thresholds, as they affect the dynamics of the system. 
All in all our results would be of interest to both the mathematical and ecological communities, and in particular to groups that are interested in conservation efforts in food chain systems, cascade control, such as in many aquatic systems, and even Allee mediated biological control.

\bibliographystyle{plain} \bibliography{REF}

\section{Appendix}
\label{app}

\subsection{ Nondimensionalisation}
\label{app1}
The model system in \eqref{eq:x1}-\eqref{eq:x3} is a nondimensionalised version of the following system
\begin{subequations} \label{eq:main}
\begin{align}
{\partial U \over {\partial T }}&=D_U\Delta U + A_1U-B_1U^2-\frac{W_1UV} {\beta_1 V+U},\\
{\partial V \over {\partial T }}&= D_V\Delta V -A_2V+\frac{W_2UV} {\beta_1 V+U}-\frac{W_3 VR} {\beta_3 R+V},\\
{\partial R \over {\partial T }}&= D_R\Delta R +R\big({{R-M}}\big)\bigg(C-\frac{W_4R} {V+A_3}\bigg).\label{eq:main1}
\end{align}
\end{subequations}

$A_1$ is the self-growth rate of the prey population $U$. $A_2$ the intrinsic death rate of the specialist predator $V$ in the absence of its only food $U$, $C$ measures the rate of self-reproduction of the generalist predator $R$. $B_1$ measures the strength of competition among individuals of the prey species $U$. $W_i's$ are the maximum values which per capita growth rate can attain. $A_3$ represents the residual loss in $R$ population due to severe scarcity of its favorite food $V$. $\beta_{1}$ is the parameters that describes the handling time of the prey $U$ by predator $V$, and $\beta_{3}$ is the parameter that describes the handling time of the prey $V$ by predator $R$.\\

 We will now go over the details of the nondimensionalisation.
We consider \eqref{eq:main}, and aim to introduce a change of variables and time scaling which reduces the number of parameters of model system \eqref{eq:main}‎.  We take
\begin{align*}
 u&={U B_1\over A_1}, \quad v={VB_1\beta_1 \over A_1},\quad r={Rb_1\beta_1\beta_3 \over A_1},\\
 t&={T \over {A_1}},\quad w_1={W_1\over \beta_1A_1},\quad a_2={A_2\over A_1},\\
 w_2&={W_2\over A_1},\quad w_3={W_3\over \beta_3A_1},\quad c={CA_1\over {A_1B_1\beta_1\beta_3}},\\
 m&={MB_1\beta_1\beta_3\over A_1},\quad D_3={A_3B_1\beta_1\over A_1},\quad w_4={W_4B_1\over {B_1^2\beta_1\beta_3^2}},\\
 d_1&={D_U\pi^2\over B_1L^2},\quad d_2={D_V\pi^2\over B_1\beta_1L^2 } ,\quad d_3={D_R\pi^2\over {B_1\beta_1\beta_3L^2}}.
\end{align*}
Considering Neuman boundary conditions, model system \eqref{eq:main} reduces to

\begin{align}
&\frac{\partial u}{\partial t}= d_1\Delta u + u-u^{2}-w_{1}\frac{uv}{u+v}, \label{eq:x0a}\\
&\frac{\partial v}{ \partial t}= d_2 \Delta v -a_{2}v+w_{2}\frac{uv}{u+v}-w_{3}\left(\frac{vr}{v+r}\right), \label{eq:x2a}\\
&\frac{\partial r}{ \partial t} = d_3 \Delta r + r\big({{r-m}}\big)\bigg(c-\frac{w_4r} {v+D_3}\bigg).\label{eq:x3a}
\end{align}
Also all parameters associated with model system (\ref{eq:x0a}-\ref{eq:x3a}) are assumed to be positive constants and have the usual biological meaning. \\

\subsection{Global existence}
\label{app0}
Here we prove proposition \ref{ge1}. The positivity of solutions follow trivially from the form of the reaction terms, which are quasi positive.
By the positivity of the first component $u(t,.)$ of the solution to \eqref{eq:x1} on $%
[0,T_{\max })$ $\times \Omega $, we get from equation \eqref{eq:x1}%
\begin{equation}
 \label{3.1}
\frac{\partial u}{\partial t}-d_{1}\Delta u\leq a_{1}u,\text{ \ on }[0,T_{\max })\times
\Omega ,  
\end{equation}%
then we use a comparison argument \cite{smoller}. That is, we can compare the solution $u$ of \eqref{3.1} to the solution $u^*$, of the linear heat equation

\begin{equation}
 \label{3.1ns}
\frac{\partial u^{*}}{\partial t}-d_{1}\Delta u^{*}= a_{1}u^{*},\text{ \ on }[0,T_{\max })\times
\Omega ,  
\end{equation}
where $u^*$ satisfies the same initial and boundary conditions as $u$.
Clearly $u \leq u^*$, and since $u^{*}$ (being the solution of a linear heat equation) is bounded, so is $u$, and we deduce

\begin{equation}
\label{3.2}
u(t,.)\leq C_{1},  
\end{equation}%
where $C_{1}$ is a constant depending only on $T_{\max }$. Then equation
\eqref{eq:x2} gives%
\begin{equation}
 \label{3.3}
\frac{\partial v}{\partial t}-d_{2}\Delta v \leq w_{1}C_{1}v, \text{ \ on }[0,T_{\max
}) \times \Omega , 
\end{equation}%
which implies by the same arguments%
\begin{equation}
\label{3.4}
v(t,.)\leq C_{2},\text{ \ on }[0,T_{\max }) \times \Omega ,  
\end{equation}%
where $C_{2}$ is a constant depending only on $T_{\max }$.
To get a bound on the component $r$ we apply Young's inequality to yield

\begin{eqnarray}
&& \frac{\partial r}{\partial t}-d_{3}\Delta r + Mcr \nonumber \\
&& =   \left( c+\frac{Mw_{3}}{v+D_{3}}\right)r^{2} -  \left( \frac{w_{3}}{v+D_{3}}\right)r^{3}     \nonumber \\
&&  \leq \left( c+\frac{Mw_{3}}{D_{3}}\right)r^{2} -  \left( \frac{w_{3}}{v+D_{3}}\right)r^{3}    \nonumber \\
&& = \left( c+\frac{Mw_{3}}{D_{3}}\right)\left( \frac{v+D_{3}}{w_{3}}\right)^{\frac{2}{3}} \left( \frac{v+D_{3}}{w_{3}}\right)^{-\frac{2}{3}} r^{2} -  \left( \frac{w_{3}}{v+D_{3}}\right)r^{3}  \nonumber \\
&& \leq \left( c+\frac{Mw_{3}}{D_{3}}\right)^{3} \left( \frac{v+D_{3}}{w_{3}}\right)^{2} +   \left( \frac{w_{3}}{v+D_{3}}\right)r^{3}  -  \left( \frac{w_{3}}{v+D_{3}}\right)r^{3}  \nonumber \\
&& \leq \left( c+\frac{Mw_{3}}{D_{3}}\right)^{3} \left( \frac{C_{2}+D_{3}}{w_{3}}\right)^{2} 
\end{eqnarray}

Here we use the bound on $v$ via \eqref{3.4}. Thus there exists a positive constant $C_{3}$ such that%
\begin{equation}
\label{3.9}
r(t,.)\leq C_{3} =  \left( c+\frac{Mw_{3}}{D_{3}}\right)^{3} \left( \frac{C_{2}+D_{3}}{w_{3}}\right)^{2} ,\text{ \ on }[0,T_{\max }) \times \Omega . 
\end{equation}%
At this step via \eqref{3.2}, \eqref{3.4}, \eqref{3.9} standard theory \cite{henry} is
applicable and the solution is global ($T_{\max }=+\infty $).
This proves the proposition.

\subsection{Apriori estimates}
\label{app2}
In all estimates made hence forth the constants $C, C_{1}, C_{2}, C_{3}, C_{\infty}$ are generic constants, that can change in value from line to line, and sometimes within the same line if so required.
We estimate the gradient of $r$ by multiplying \eqref{eq:x3} by $-\Delta r$, and integrating by parts and using standard methods to obtain 

\begin{equation}
\label{eq:x11}
\frac{1}{2}\frac{d}{dt}||\nabla r||^{2}_{2}  + d_{3}||\Delta r||^{2}_{2} = \int_{\Omega}r(r-m)(c-w_{4}\frac{r}{v+D_{3}})(-\Delta r)dx.
\end{equation}

We have to estimate $\int_{\Omega}r(r-m)(c-w_{4}\frac{r}{v+D_{3}})(-\Delta r)dx$. Note

\begin{align*}
\int_{\Omega}r(r-m)(c-\frac{r}{v+D_{3}})(-\Delta r)dx
& = \int_{\Omega} \left( \left( c+\frac{mw_{4}}{v+D_{3}}\right)r^{2} - \frac{w_{4}}{v+D_{3}}r^{3} - mcr  \right) (-\Delta r) dx, \nonumber \\
& = \int_{\Omega} \left( \left( c+\frac{mw_{4}}{v+D_{3}}\right)r^{2} - \frac{w_{4}}{v+D_{3}}r^{3} \right)   (-\Delta r) dx - cm||\nabla r||^{2}_{2}. \nonumber \\
\end{align*}

Thus we obtain

\begin{equation}
\frac{1}{2}\frac{d}{dt}||\nabla r||^{2}_{2}  + d_{3}||\Delta r||^{2}_{2} + cm||\nabla r||^{2}_{2}  =   \int_{\Omega} \left( \left(c+\frac{mw_{4}}{v+D_{3}}\right)r^{2} - \frac{w_{4}}{v+D_{3}}r^{3}  \right)  (-\Delta r) dx.
\end{equation}

We now proceed in 2 steps. We first assume $\left( \left(c+\frac{mw_{4}}{v+D_{3}}\right)r^{2} - \frac{w_{4}}{v+D_{3}}r^{3} \right) \geq 0$.

Young's Inequality with epsilon then gives us the following estimate 

\begin{align}
\label{eq:x11y}
\left( \left(c+\frac{mw_{4}}{v+D_{3}}\right)r^{2} - \frac{w_{4}}{v+D_{3}}r^{3} \right)&\leq\left( \frac{v+D_{3}}{w_{4}}\right)^{\frac{1}{3}} \left(c+\frac{mw_{4}}{v+D_{3}}\right)r \left(\frac{w_{4}}{v+D_{3}}\right)^{\frac{1}{3}}r -w_{4}\frac{r^{3}}{v+D_{3}},\nonumber \\
& \leq \left(\frac{v+D_{3}}{w_{4}}\right)^{\frac{4}{9}} \left(c+\frac{mw_{4}}{v+D_{3}}\right)^{\frac{4}{3}}r^{\frac{4}{3}} + w_{4}\frac{r^{3}}{v+D_{3}} -w_{4}\frac{r^{3}}{v+D_{3}}, \nonumber \\
& \leq  \left(\frac{||v+D_{3}||_{\infty}}{w_{4}}\right)\left(c+\frac{mw_{4}}{v+D_{3}}\right)^{\frac{4}{3}}r^{\frac{4}{3}} + w_{4}\frac{r^{3}}{v+D_{3}} -w_{4}\frac{r^{3}}{v+D_{3}} , \nonumber \\
&\leq  \left(\frac{C+D_{3}}{w_{4}}\right) \left(c+\frac{mw_{4}}{D_{3}}\right)^{\frac{4}{3}}r^{\frac{4}{3}} +  w_{4}\frac{r^{3}}{v+D_{3}} -w_{4}\frac{r^{3}}{v+D_{3}},\nonumber \\
& =  \left(\frac{C+D_{3}}{w_{4}}\right) \left(c+\frac{mw_{4}}{D_{3}}\right)^{\frac{4}{3}}r^{\frac{4}{3}}.  \nonumber \\
\end{align}

Then we obtain

\begin{eqnarray}
\label{eq:x11nn}
\frac{1}{2}\frac{d}{dt}||\nabla r||^{2}_{2}  + d_{3}||\Delta r||^{2}_{2} + 2cm||\nabla r||^{2}_{2}&=& \int_{\Omega}\left( \left(c+\frac{mw_{4}}{v+D_{3}}\right)r^{2} - \frac{w_{4}}{v+D_{3}}r^{3}  \right)(-\Delta r)dx, \nonumber \\
&\leq&  \left(\frac{C+D_{3}}{w_{4}}\right) \left(c+\frac{mw_{4}}{D_{3}}\right)^{\frac{4}{3}}\int_{\Omega}r^{\frac{4}{3}}|-\Delta r| dx\nonumber \\
\end{eqnarray}
 This follows via \eqref{eq:x11y}. Thus we obtain

\begin{equation}
\label{eq:x11}
\frac{1}{2}\frac{d}{dt}||\nabla r||^{2}_{2}  + d_{3}||\Delta r||^{2}_{2}  \leq C_{1}||r||^{\frac{8}{3}}_{\frac{8}{3}} + \frac{d_{3}}{2}||\Delta r||^{2}_{2}.
\end{equation}

Here $C_{1}=\left(\frac{C+D_{3}}{w_{4}}\right) \left(c+\frac{w_{4}}{D_{3}}\right)^{\frac{4}{3}}$.
Now we can use the sobolev embedding of $H^1(\Omega) \hookrightarrow L^{\frac{8}{3}}(\Omega)$, where $C_{3}$ is the embedding constant, to obtain

\begin{equation}
\label{eq:x11}
\frac{d}{dt}||\nabla r||^{2}_{2}  + d_{3}||\Delta r||^{2}_{2}   \leq C_{3}C_{1}\left(||\nabla r||^{2}_{2}\right)\left(||\nabla r||^{2}_{2}\right).
\end{equation}

Thus

\begin{equation}
\label{eq:x11}
\frac{d}{dt}||\nabla r||^{2}_{2}  \leq C_{3}C_{1}(||\nabla r||^{2}_{2})(||\nabla r||^{2}_{2}).
\end{equation}

Now we invoke the estimate via \eqref{eq:x11n23n}, and the uniform Gronwall lemma to obtain

\begin{equation}
\label{eq:f1r4}
 \mathop{\limsup}_{t \rightarrow \infty}||\nabla r||^{2}_{2}  \leq C.
\end{equation}

Now assume $\left( \left(c+\frac{mw_{4}}{v+D_{3}}\right)r^{2} - \frac{w_{4}}{v+D_{3}}r^{3} \right) < 0$.
Then we obtain

\begin{eqnarray}
\label{eq:x11nn}
&&\frac{1}{2}\frac{d}{dt}||\nabla r||^{2}_{2}  + d_{3}||\Delta r||^{2}_{2} + cm||\nabla r||^{2}_{2},\nonumber \\
&\leq& \int_{\Omega}\left(c+\frac{mw_{4}}{D_{3}}\right)r^{2}(-\Delta r)dx, \nonumber \\
&\leq&  \frac{2}{d_{3}}\left(c+\frac{mw_{4}}{D_{3}}\right)||r||^{4}_{4} + \frac{d_{3}}{2}||\Delta r||^{2}_{2}. \nonumber \\
\end{eqnarray}

We now use the Sobolev embedding of $H^1(\Omega) \hookrightarrow L^{4}(\Omega)$ to obtain
\begin{equation}
\label{eq:x1mn}
\frac{d}{dt}||\nabla r||^{2}_{2}  + d_{3}||\Delta r||^{2}_{2} + cm||\nabla r||^{2}_{2} \leq \frac{2}{d_{3}}\left(c+\frac{w_{4}}{D_{3}}\right) ||\nabla r||^{4}_{2},
\end{equation}

dropping the $cm||\nabla r||^{2}_{2}$ term yields

\begin{equation}
\frac{d}{dt}||\nabla r||^{2}_{2}  \leq \frac{2}{d_{3}}\left(c+\frac{w_{4}}{D_{3}}\right) \left(||\nabla r||^{2}_{2}\right)\left(||\nabla r||^{2}_{2}\right),
\end{equation}

\begin{equation}
\label{eq:f1d4nn}
 \mathop{\limsup}_{t \rightarrow \infty} ||\nabla r||^{2}_{2}  \leq C, 
\end{equation}

\subsection{Apriori estimates independent of the parameter $m$}
\label{app3}
Recall the following $L^2$ estimate

\begin{equation*}
\frac{1}{2}\frac{d}{dt}||r||^{2}_{2}  + d_{3}||\nabla r||^{2}_{2} + cm||r||^{2}_{2} + \int_{\Omega}\frac{w_4}{v+D_{3}}r^{4}dx \leq \int_{\Omega}\left(c+\frac{mw_{4}}{v+D_{3}} \right)r^{3} dx.
\end{equation*}

We then use H\"{o}lder's inequality followed by Young's inequality  to obtain

\begin{eqnarray}
\label{eq:x11r}
&&\frac{1}{2}\frac{d}{dt}||r||^{2}_{2}  + d_{3}||\nabla r||^{2}_{2} + cm||r||^{2}_{2} + \int_{\Omega}\frac{w_4}{v+D_{3}}r^{4}dx,  \nonumber \\
&& \leq \frac{3}{4}\int_{\Omega}\frac{w_4}{v+D_{3}}r^{4}dx + |\Omega|^{\frac{1}{4}}\left(c+\frac{w_{4}}{D_{3}}\right) \left(\frac{w_{4}}{D_{3}}\right)^{\frac{1}{4}}.\nonumber \\
\end{eqnarray}

Now we drop the $ cm||r||^{2}_{2}$ term from the left hand side, to avoid the singularity caused by $\frac{1}{m}$, the $d_{3}||\nabla r||^{2}_{2}$, and use the estimate on $||v||_{\infty}$ via \eqref{eq:liea}, and the embedding of $L^{4}(\Omega) \hookrightarrow L^{2}(\Omega)$ to obtain

\begin{equation}
\frac{d}{dt}||r||^{2}_{2} +  \frac{Cw_4}{C+D_{3}}||r||^{2}_{2} \leq    2|\Omega|^{\frac{1}{4}}\left(c+\frac{w_{4}}{D_{3}}\right) \left(\frac{w_{4}}{D_{3}}\right)^{\frac{1}{4}}.
\end{equation}

Here, we assume $||r||_{2} \geq 1$, else we already have a bound for $||r||_{2} $. 
Next we use Gronwall's inequality to obtain

\begin{equation}
||r||^{2}_{2}  \leq  e^{-\left(  \frac{Cw_4}{C+D_{3}}\right)t}||r_{0}||^{2}_{2} + \frac{2|\Omega|^{\frac{1}{4}}\left(c+\frac{w_{4}}{D_{3}}\right) \left(\frac{w_{4}}{D_{3}}.\right)^{\frac{1}{4}}(C+D_{3})}{Cw_{4}}.
\end{equation}

Thus for $t  > t_{0}= \max(t^{*}+1,\frac{ \ln(||r_{0}||^{2}_{2})}{\left(  \frac{Cw_4}{C+D_{3}}\right)})$, we obtain

\begin{equation}
\label{eq:rl2n}
||r||^{2}_{2}  \leq  1 + \frac{2|\Omega|^{\frac{1}{4}}\left(c+\frac{w_{4}}{D_{3}}\right) \left(\frac{w_{4}}{D_{3}}\right)^{\frac{1}{4}}(C+D_{3})}{Cw_{4}},
\end{equation}

which is a uniform $L^2(\Omega)$ bound in $r$ that is independent of the Allee parameter $m$, time and initial data.

We next focus on making $H^{1}(\Omega)$ estimates, that are independent of the Allee parameter $m$.

We  choose the estimate derived in \eqref{eq:rl2n} for $||r||_{2}$ and insert this in \eqref{eq:x11n23n} to obtain

\begin{eqnarray}
\label{eq:nh1na}
&& \int^{t+1}_{t}||\nabla r||^{2}_{2} ds \nonumber \\
&& \leq  1 + \frac{2|\Omega|^{\frac{1}{4}}\left(c+\frac{w_{4}}{D_{3}}\right) \left(\frac{w_{4}}{D_{3}}\right)^{\frac{1}{4}}(C+D_{3})}{Cw_{4}}+ 2|\Omega|^{\frac{1}{4}}\left(c+\frac{w_{4}}{D_{3}}\right) \left(\frac{w_{4}}{D_{3}}\right)^{\frac{1}{4}},  \nonumber \\
&&  \ \mbox{for} \ t > t_{0}.
\end{eqnarray}

Now we choose 

\begin{equation}
\label{eq:kn1m}
K_{m}=\max \left(\frac{2}{d_{3}}\left(c+\frac{mw_{4}}{D_{3}}\right), C_{3}\left(\frac{C+D_{3}}{w_{4}}\right) \left(c+\frac{mw_{4}}{D_{3}}\right)^{\frac{4}{3}} \right).
\end{equation}

Then, we can majorise the right hand side of the above by plugging in $m=1$. Since the above quantities are not singular in $m$, this is possible. Thus we obtain

\begin{equation}
K=\max \left(\frac{2}{d_{3}}\left(c+\frac{w_{4}}{D_{3}}\right), C_{3}\left(\frac{C+D_{3}}{w_{4}}\right) \left(c+\frac{w_{4}}{D_{3}}\right)^{\frac{4}{3}} \right),
\end{equation}

going back to the $H^1(\Omega)$ estimates in \eqref{eq:f1r4}, and \eqref{eq:f1d4nn}  one obtains,

\begin{equation}
\label{eq:nr1}
\frac{d}{dt}||\nabla r||^{2}_{2}  \leq K  \left(||\nabla r||^{2}_{2} \right) \left(||\nabla r||^{2}_{2}\right).
\end{equation}

Also note the integral in time estimate via \eqref{eq:nh1na} is now independent of $m$.
This allows us to apply the uniform Gronwall lemma on \eqref{eq:nr1} with $\beta(t) = ||\nabla r||^{2}_{2}, \ \zeta(t)=K||\nabla r||^{2}_{2} , \ h(t)= 0, q=1$, and the estimate via \eqref{eq:nh1na} to yield the following bound

\begin{eqnarray}
\label{eq:nrn}
&&||\nabla r||^{2}_{2}  \nonumber \\
&&\leq C_{K} e^{K C_{K}}, \ \mbox{for} \ t > t_{1}=t_{0} + 1.
\end{eqnarray}

Thus the $H^{1}(\Omega)$ estimate for $r$ can be made independent of $m$. 

\subsection{Upper semi-continuity of global attractors}
\label{app4}
In this section we prove upper semi continuity of global attractors.
We first introduce certain concepts pertinent to the upper semicontinuity of attractors.

\begin{definition}[Uniform dissipativity]
Suppose there is a family of semiflows $\{ \{  S_{\lambda}(t) \}_{t \geq 0}\}_{\lambda \in \Lambda}$ on a Banach space $H$, where $\Lambda$ is an open set in a Euclidean space of parameter $\lambda$, is called uniformly dissipative at $\lambda_{0} \in \Lambda$, if there is a neighborhood $\mathcal{N}$ of $\lambda_{0}$ in $\Lambda$ and there is a bounded set $\mathcal{B} \subset H$ such that $\mathcal{B}$ is an absorbing set for each semiflow $S_{\lambda}(t)$, $\lambda \in \mathcal{N}$, in common.
   \end{definition}

\begin{definition}[Upper semicontinuity]
A family of semiflows $\{ \{  S_{\lambda}(t) \}_{t \geq 0}\}_{\lambda \in \Lambda}$ on a Banach space $H$, where $\Lambda$ is an open set in a Euclidean space of parameter $\lambda$, and that there exists a global attractor $\mathcal{A}_{\lambda}$ in $\mathcal{H}$ for each semiflow $\{ \{  S_{\lambda}(t) \}_{t \geq 0}\}_{\lambda \in \Lambda}$. If $\lambda_{0} \in \Lambda$ and

 \begin{equation*}
    dist_{\mathcal{H}}(\mathcal{A}_{\lambda},\mathcal{A}_{\lambda_{0}}) \rightarrow 0,  \ \mbox{as} \  \lambda \rightarrow \lambda_{0} \ \mbox{in} \ \Lambda,
    \end{equation*}

with respect to the Hausdorff semidistance, then we say that the family of global attractors $\{\mathcal{A}_{\lambda} \}_{\lambda \in \Lambda}$ is upper semicontinuous at $\lambda_{0}$, or that $\mathcal{A}_{\lambda} $ is \emph{robust} at $\lambda_{0}$.
 \end{definition}

We next recall the following lemma

\begin{lemma}[Gronwall-Henry Inequality]

Let $\psi(t)$ be a nonnegative function in $L^{\infty}_{loc}[0,T;\mathbb{R})$ and $\zeta(.)$ $\in$ $L^{1}_{loc}[0,T;\mathbb{R})$, such that the following inequality is satisfied:
\begin{equation}
\psi(t) \leq \zeta(t) + \mu\int^{t}_{0} (t-s)^{r-1}\psi(s)ds, \ t \in (0,T),
\end{equation}
where $0 < T \leq \infty$, and $\mu, r$ are positive constants. Then it holds that
\begin{equation}
\psi(t) \leq \zeta(t) + \kappa\int^{t}_{0} \Phi(\kappa(t-s))\psi(s)ds, \ t \in (0,T), 
\end{equation}
where $\kappa = (\mu\Gamma(r))^{\frac{1}{r}}$, $\Gamma(.)$ is the gamma function, and the function $\phi(t)$ is given by

\begin{equation}
\phi(t)= \sum^{\infty}_{n=1} \frac{nr}{\Gamma(nr+1)} t^{nr-1} = \sum^{\infty}_{n=1} \frac{nr}{\Gamma(nr)} t^{nr-1}.
\end{equation}
\end{lemma}

We state the following theorem

\begin{theorem}
\label{thm:gaus1}
Consider the reaction diffusion equation described via \eqref{eq:x1}-\eqref{eq:x3} where $\Omega$ is of spatial dimension $n=1, 2, 3$. There exists a universal constant $K_{\infty} > 0$, independent of the Allee parameter $m$ that bounds uniformly in $L^{\infty}$ the family of global attractors $\mathcal{A}_{m}$. That is,

\begin{equation}
\bigcup_{0\leq m \leq 1} \mathcal{A}_{m} \subset B_{L^{\infty}}(0,K_{\infty}).
\end{equation}

Where $B_{L^{\infty}}(0,K_{\infty})$ is the closed ball of radius $K_{\infty}$ in the space $L^{\infty}(\Omega)$, with $K_{\infty}=3C$, where the $C$ is from \eqref{eq:liea}.
\end{theorem}

This follows from the estimates via \eqref{eq:liea}.

We next focus on the uniform dissipativity in $L^{2}(\Omega)$.
The uniform in parameter $m$ $L^2(\Omega)$ estimate via \eqref{eq:rl2n} enables us to state the following theorem

\begin{theorem}
\label{thm:gaus1l2}
Consider the reaction diffusion system described via \eqref{eq:x1}-\eqref{eq:x3} where $\Omega$ is of spatial dimension $n=1, 2, 3$. The family of semiflows 
$\{ \{  S_{m}(t) \}_{t \geq 0}\}_{m \geq 0}$ for this system on $H$, is uniformly dissipative at $m=0$. Specifically there exists a constant $K_{H} > 0$ such that the ball $ B_{H}(0,K_{H})$,  which is the closed ball of radius $K_{H}$ in the space $H$ is a common absorbing set for the semiflows $\{ \{  S_{m}(t) \}_{t \geq 0}\}$ for all $m \in [0,1]$. Here $K_{H}=1 + \frac{2|\Omega|^{\frac{1}{4}}\left(c+\frac{w_{4}}{D_{3}}\right) \left(\frac{w_{4}}{D_{3}}\right)^{\frac{1}{4}}(C+D_{3})}{Cw_{4}}+ 2C$, where the $2C$ comes from lemma \ref{lem:lemba}.
\end{theorem}

Next, let us choose $(u_{0},v_{0},r_{0})$  $\in$ $\mathcal{U}$. We know $\mathcal{U} \subset B_{H}(0,K_{H})$, the common absorbing ball in $H$ for the semiflow $S_{m}(t)$, $m \in [0,1]$, from theorem \ref{thm:gaus1l2}.
We use the earlier $H^1(\Omega)$ and $L^2(\Omega)$ estimates on $u,v$ from lemma \ref{lem:lemba}, noticing that they do not depend on $m$, we also use \eqref{eq:rl2n} and \eqref{eq:nrn} to obtain

\begin{eqnarray}
&&||S_{m}(t)((u_{0},v_{0},r_{0}))||^{2}_{E}  \nonumber \\
&&= ||(\nabla u,\nabla v,\nabla r)||^2 + ||( u, v, r)||^2  \nonumber \\
&& \leq 4C + C_{K} e^{K C_{K}} +  1 + \frac{2|\Omega|^{\frac{1}{4}}\left(c+\frac{w_{4}}{D_{3}}\right) \left(\frac{w_{4}}{D_{3}}\right)^{\frac{1}{4}}(C+D_{3})}{Cw_{4}}.
\end{eqnarray}

This is true for any $(u_{0},v_{0},r_{0}) \in \mathcal{U}$ and $m \in [0,1]$. Now by invariance $S_{m}(t)\mathcal{A}_{m}=\mathcal{A}_{m}$, and so $\mathcal{A}_{m} \subset B_{E}(0,K_{E})$. Where 

\begin{equation}
\label{eq:Ke1}
K_{E} = 4C + C_{K} e^{K C_{K}} +  1 + \frac{2|\Omega|^{\frac{1}{4}}\left(c+\frac{w_{4}}{D_{3}}\right) \left(\frac{w_{4}}{D_{3}}\right)^{\frac{1}{4}}(C+D_{3})}{Cw_{4}}.
\end{equation}

This yields the following theorem 

\begin{theorem}
\label{thm:gaus1}
Consider the reaction diffusion equation described via \eqref{eq:x1}-\eqref{eq:x3} where $\Omega$ is of spatial dimension $n=1, 2, 3$. There exists a universal constant $K_{E} > 0$, independent of the Allee parameter $m$ that bounds uniformly in $E$ the family of global attractors $\mathcal{A}_{m}$. That is,

\begin{equation}
\bigcup_{0\leq m \leq 1} \mathcal{A}_{m} \subset B_{E}(0,K_{E}),
\end{equation}

where $ B_{E}(0,K_{E})$ is the closed ball of radius $K_{E}$ explicitly given in \eqref{eq:Ke1}.
\end{theorem}

We next show that the trajectories $S_{m}(t)U, t\geq 0$, are uniformly $E$-bounded, with respect to $m$. Since the estimate in \eqref{eq:Ke1} does not depend on $m$, we can take a supremum over $m \in [0,1]$ and still achieve the same bound. Thus we have the following result

\begin{theorem}
\label{thm:gaus12}
Consider the reaction diffusion equation described via \eqref{eq:x1}-\eqref{eq:x3}. There exists a universal constant $K_{E} > 0$, independent of the Allee parameter $m$ such that

\begin{equation}
\sup_{0\leq m \leq 1} \sup_{t \geq 0} S_{m}(t)\left(\bigcup_{0\leq m \leq 1} \mathcal{A}_{m} \right) \subset B_{E}(0,K_{E}),
\end{equation}

where, $ B_{E}(0,K_{E})$ is the closed ball of radius $K_{E}$ in the space $E$.
\end{theorem}

We next set $m=0$ in \eqref{eq:x3} and note that the corresponding estimates in lemma \ref{lem:lemba} hold for $u, v$, and do not depend on the parameter $m$. Thus we carry out the analysis as in \eqref{eq:nh1na} - \eqref{eq:Ke1} to still yield

\begin{align}
||\nabla r||^{2}_{2} \leq  C_{K} e^{K C_{K}} , \ \mbox{for} \ t > t_{1} + 1. 
\end{align}

Where $C_{K} e^{K C_{K}}$ is independent of the parameter $m$. Thus we can state the following result

\begin{theorem}
\label{thm:ga1m0}
Consider the reaction diffusion equation described via \eqref{eq:x1}-\eqref{eq:x3}, when $m=0$, and where $\Omega$ is of spatial dimension $n=1, 2, 3$. There exists a $(H,E)$ global attractor $\mathcal{A}$ for the system. This is compact and invariant in $H$, and it attracts all bounded subsets of $H$ in the $E$ metric.
\end{theorem}

The proof of the above theorem follows essentially via theorem \ref{thm:ga2}.

Note, since $\left(\bigcup_{0\leq m \leq 1} \mathcal{A}_{m} \right)$ is a bounded set in $H$, we can state the following result ,

\begin{theorem}
\label{thm:gaus13}
Consider the reaction diffusion equation described via \eqref{eq:x1}-\eqref{eq:x3}, with $m=0$. The global attractor for this system $\mathcal{A}_{0}$, attracts the set $\left(\bigcup_{0\leq m \leq 1} \mathcal{A}_{m} \right) $, with respect to the $E$-norm. That is

\begin{equation}
\lim_{ t \rightarrow \infty} dist_{E} \left(S_{0}(t)\left(\bigcup_{0\leq m \leq 1} \mathcal{A}_{m}\right), \mathcal{A}_{0}  \right).
\end{equation}
\end{theorem}

We now use the Gronwall-Henry lemma, to show uniform convergence in the parameter $m$.

\begin{theorem}
\label{thm:gaum}
Consider the reaction diffusion equation described via \eqref{eq:x1}-\eqref{eq:x3}, and associated semi-group $S_{m}(t)$. Also consider \eqref{eq:x1}-\eqref{eq:x3}, when $m=0$, that is with associated semi-group $S_{0}(t)$. For any given $t \geq 0$, it holds that

\begin{equation}
\sup_{g_{0} \in \mathcal{U}} ||S_{m}(t)g_{0} - S_{0}(t)g_{0}||_{E} \rightarrow 0, \mbox{as} \ m \rightarrow 0^{+} . 
\end{equation}

\end{theorem}

For any given initial data $g_{0} \in \mathcal{U}$, we note $S_{m}(t)g_{0}$ and $S_{0}(t)g_{0}$, are both classical solutions, hence mild solutions. Thus we denote

\begin{equation}
w(t) = S_{m}(t)g_{0} - S_{0}(t)g_{0},  t \geq 0, \ \mbox{with} \ w(0)=0.
\end{equation}

Using the form of mild solutions, we can write down the following,

\begin{equation}
w(t) = \int^{t}_{0}e^{A(t-\sigma)}[f_{0}(S_{m}(\sigma)g_{0}) - f_{0}(S_{0}(\sigma)g_{0})] d \sigma + m \int^{t}_{0}e^{A(t-\sigma)}h(S_{m}(\sigma)g_{0})  d \sigma, \ t \geq 0.
\end{equation}

Here $e^{A t}$ , $t \geq 0$ is the $C_{0}$ semi-group generated by $A : D(A) \rightarrow H$. Note

\begin{equation}
f_{m}(u,v,r)=
\left(
 \begin{array}{c}                                                                             
u-u^{2}-w_{1}\frac{uv}{u+v}\\                                                                               
  -a_{2}v+w_{2}\frac{uv}{u+v}-w_{3}\left(\frac{vr}{v+r}\right) \\
-\frac{w_4} {v+D_3}r^{3} + cr^{2} \\                                                                               
  \end{array} \ 
 \right),    
h(r) = \left(
 \begin{array}{c}    
0\\                                                                               
 0 \\
-mcr+\frac{mw_4} {v+D_3}r^{2} \\                                                                               
  \end{array}   
 \right).    
\end{equation}

The above is easily seen to be Lipschitz continuous in $E$, due to the Sobolev embedding of $H^{1}(\Omega) \hookrightarrow L^{6}(\Omega)$. 
Here $f_{0}$, which is gotten by setting $m=0$ in the above, is also Lipschitz continuous, thus there is a Lipschitz constant $L_{f_{0}}(K_{E}) > 0$, depending only on $K_{E}$, where $K_{E}$ is given in \eqref{eq:Ke1}, s.t

\begin{equation}
||f_{0}(g_{1}) - f_{0}(g_{2})|| \leq L_{f_{0}}(K_{E}) || g_{1} - g_{2} ||_{E} ,  
\end{equation}

for any $g_{1}, g_{2} \in B_{E}(0,K_{E})$. From theorems \ref{thm:gaus1}, \ref{thm:gaus12} we have

\begin{eqnarray}
&& ||w(t)||_{E}  \nonumber \\
&& \leq \int^{t}_{0}||e^{A(t-\sigma)}||_{\mathcal{L}(H,E)} L_{f_{0}}(K_{E}) || S_{m}(\sigma)g_{0} -S_{0}(\sigma)g_{0}||_{E} d \sigma,   \nonumber \\
&& + m \int^{t}_{0}||e^{A(t-\sigma)}||_{\mathcal{L}(H,E)}h(S_{m}(\sigma)g_{0})  d \sigma ,\nonumber \\
&& \leq Cm(K_{E})^{3} \int^{t}_{0} N(t-\sigma)^{-\frac{1}{2}} d \sigma + NL_{f_{0}}(K_{E}) \int^{t}_{0} (t-\sigma)^{-\frac{1}{2}} ||w(t)||_{E} d \sigma, \nonumber \\
&& \leq Cm(K_{E})^{3}t^{\frac{1}{2}} + NL_{f_{0}}(K_{E}) \int^{t}_{0} (t-\sigma)^{-\frac{1}{2}} ||w(t)||_{E} d \sigma.
\end{eqnarray}

This follows via standard decay estimates on the semi-group \cite{SY02}, theorems \ref{thm:gaus1}, \ref{thm:gaus12} and the Sobolev embedding of $H^{1}(\Omega) \hookrightarrow L^{6}(\Omega)$. 
We now apply the Gronwall-Henry inequality with $\Psi(t)= ||w(t)||_{E},  \zeta(t)=Cm(K_{E})^{3}t^{\frac{1}{2}}, \mu =  NL_{f_{0}}(K_{E}), r=\frac{1}{2}$ to obtain

\begin{equation}
||S_{m}(t)g_{0} - S_{0}(t)g_{0}||_{E} = ||w(t)||_{E}  \leq m \left(\Psi(t) + k\int^{t}_{0}\Phi(k(t-s))\Psi(s)ds \right), \ t \geq 0.
\end{equation}
 
Here 

\begin{equation}
\zeta(t) = C(K_{E})^{3}t^{\frac{1}{2}}, \ k=NL_{f_{0}}(K_{E})\left(\Gamma \left(\frac{1}{2}\right)\right)^{\frac{1}{2}} , \ \Phi(t)= \sum^{\infty}_{n=1}\frac{1}{\Gamma(\frac{n}{2})}t^{\frac{n}{2}-1}.
\end{equation}

Thus we get the uniform convergence as required, and

\begin{equation}
\sup_{g_{0} \in \mathcal{U}} ||S_{m}(t)g_{0} - S_{0}(t)g_{0}||_{E} = ||w(t)||_{E}\rightarrow 0, \mbox{as} \ m \rightarrow 0^{+} . 
\end{equation}
 where the convergence is uniform in the parameter $m$.

\end{document}